\documentclass[10pt,journal]{IEEEtran}
\usepackage[margin=0.75in]{geometry}
\usepackage{amsmath,amssymb, ntheorem}
\usepackage{tikz}
\usepackage{epstopdf}
\usepackage{upgreek}
\usepackage{subcaption}
\usepackage{stmaryrd}
\usepackage[font=small,labelfont=bf,belowskip=0pt]{caption}
\usetikzlibrary{automata,positioning,fit,shapes.geometric}
\pgfdeclarelayer{bg}    
\pgfsetlayers{bg,main}  
\pgfdeclarelayer{background}
\pgfdeclarelayer{foreground}
\pgfsetlayers{background,main,foreground}
\usepackage{ifthen}
\usepackage{psfrag}
\usepackage{cite}
\usepackage{graphicx}
\usepackage{graphics}
\usepackage{epstopdf}
\usepackage{bbm}
\usepackage{textcomp}
\usepackage{mathabx}
\usepackage{verbatim}
\usepackage{accents}
\usepackage{algorithm}
\usepackage{algorithmicx}
\usepackage{algpseudocode}
\usepackage{float}
\usepackage{framed}
\usepackage{color,xspace}
\usepackage{tikz}
\usetikzlibrary{decorations.pathreplacing}
\usepackage{upgreek}
\usetikzlibrary{automata,positioning}
\usepackage{pstool}
\usepackage{url}

\newcommand{\Dists}{\Psi}

\newcommand{\curlyleq}{\preccurlyeq}

\newcommand{\maxiter}{k_{\text{max}}}

\newcommand{\ubar}[1]{\underaccent{\bar}{#1}}
\newcommand{\utilde}[1]{\underaccent{\tilde}{#1}}

\newcommand{\real}{\mathbb{R}}
\newcommand{\realnonnegative}{\mathbb{R}_{\geq 0}}
\newcommand{\StabFun}{J}

\newcommand{\oprocendsymbol}{\hbox{$\bullet$}}
\newcommand{\oprocend}{\relax\ifmmode\else\unskip\hfill\fi\oprocendsymbol}
\newcommand{\proofend}{\relax\ifmmode\else\unskip\hfill\fi$\square$}

\usepackage[normalem]{ulem}
\makeatletter
\renewtheoremstyle{plain}
{\item{\theorem@headerfont ##1\ ##2\theorem@separator}~}
{\item{\theorem@headerfont ##1\ ##2\ (##3)\theorem@separator}~}
\makeatother

{\theoremheaderfont{\upshape\bfseries}
	\theorembodyfont{\normalfont\slshape}
	\newtheorem{remark}{Remark}
	\newtheorem{defn}{Definition}
	\newtheorem{lem}{Lemma}
	
	\newtheorem{theorem}{Theorem}
	
	\newtheorem{prop}{Proposition}
}

\newenvironment{proof}{\noindent \textbf{Proof :}}{}

\newcommand{\Neighbors}{\mathcal{N}}

\newcommand{\Ito}{It\^{o} }
\newcommand{\Itos}{It\^{o}'s }
\newcommand{\deriv}{\mathrm{d}}

\newcommand{\ddtop}[1]{\frac{\deriv #1}{\deriv \mathrm{t}}}
\newcommand{\E}{\mathbbm{E}}
\newcommand{\dP}{\deriv \mathbbm{P}}
\newcommand{\dQ}{\deriv \mathbbm{Q}}
\newcommand{\dR}{\deriv \mathbbm{R}}
\newcommand{\dS}{\deriv \mathbbm{S}}
\newcommand{\dU}{\deriv \mathbbm{U}}
\newcommand{\dV}{\deriv \mathbbm{V}}
\newcommand{\Nodes}{\mathcal{V}}
\newcommand{\FrechetLower}{\mathcal{F}_{\ell}}
\newcommand{\FrechetUpper}{\mathcal{F}_{u}}
\newcommand{\Frechet}{Fr\'{e}chet }

\newcommand{\Boundupri}[1]{\bar{\mathcal{B}}_{i}^{#1}}

\newcommand{\Exp}[1]{e^{#1}}
\newcommand{\Cost}{\mathcal{C}}

\newcommand{\Graph}{\mathcal{G}}
\newcommand{\Edges}{\mathcal{E}}

\newcommand{\BetaIJmax}{\bar{\beta}_{ij}}
\newcommand{\GammaIJmax}{\bar{\gamma}_{ij}}
\newcommand{\NP}{\mathcal{NP}}

\newcommand{\CandidateSolution}{\mathbb{A}^\star}

\newcommand{\umpc}{u_{\text{mpc}}}
\newcommand{\uaux}{u_{\text{aux}}}

\newcommand{\me}{\mathrm{e}}
\newcommand{\telim}{\tau_{\text{elim}}}
\newcommand{\StateSpace}{\mathcal{X}}
\newcommand{\IdxSet}{\mathcal{I}}

\newcommand{\delt}{\Delta t}

\newcommand{\tone}{\tau_{1}}
\newcommand{\Zplus}{\mathbb{Z}_{\geq 0}}
\newcommand{\InfectExposed}{\ell}
\newcommand{\SamplingTimes}{T_{\delt}}

\newcommand{\Muller}{M{\"u}ller }
\newcommand{\Erdos}{Erd{\"o}s}

\newcommand{\ActionSet}{\mathcal{A}}
\newcommand{\action}{a}
\newcommand{\utotal}{u_{\text{tot}}}

\newcommand{\contratime}{t_{\text{c}}}
\newcommand{\precontratime}{t_{\text{c}}^{-}}

\newcommand{\MFpapers}{\cite{Preciado2014,Nowzari2017,Wang2014,Ogura2016,Lee2016a,Hota2017,Pare2017,Watkins2015,Eshghi2016,Eshghi2017}}

\pgfdeclarelayer{bg}
\pgfdeclarelayer{bbg}
\pgfsetlayers{bbg,bg,main}

\usetikzlibrary{shapes,arrows}
\tikzset{base/.style={draw, align=center, minimum height=4ex},
	test1/.style={base, diamond, aspect=2, text width=5em, inner sep=5pt},
	test2/.style={base, diamond, aspect=2, text width=5em, inner sep=-5pt}
}
\tikzstyle{block} = [draw, fill=blue!20, rectangle, minimum height=2em, minimum width=4em]
\tikzstyle{blkdiamond} = [draw, fill=blue!20, diamond]
\tikzstyle{blkcloud} = [draw, fill=blue!20, cloud, draw,cloud puffs=10,cloud puff arc=120, aspect=1.5, inner ysep=0.2em]
\tikzstyle{sum} = [draw, fill=blue!20, circle, node distance=1cm]
\tikzstyle{disk} = [draw, fill=blue!20, circle, node distance=1cm, minimum height=2em]
\tikzstyle{input} = [coordinate]
\tikzstyle{output} = [coordinate]
\tikzstyle{pinstyle} = [pin edge={to-,thin,black}]

\begin{document}
\title{\Huge Robust Economic Model Predictive Control\textit{} of \\ Continuous-time Epidemic Processes}      

\author{Nicholas J. Watkins, Cameron Nowzari, and George J. Pappas\thanks{N.J. Watkins and G.J. Pappas are with the Department of Electrical and Systems Engineering, University of Pennsylvania, Philadelphia, PA 19104, USA, {\tt\small \{nwatk,pappasg\}@upenn.edu}.  C. Nowzari is with the Department of Electrical and Computer Engineering, George Mason University, Fairfax, VA  22030, USA, {\tt\small cnowzari@gmu.edu}.	
}}

\maketitle
\begin{abstract}
	In this paper, we develop a robust economic model predictive controller for the containment of stochastic Susceptible-Exposed-Infected-Vigilant $(SEIV)$ epidemic processes which drives the process to extinction quickly, while minimizing the rate at which control resources are used.
	The work we present here is significant in that it addresses the problem of efficiently controlling general stochastic epidemic systems without relying on mean-field approximation, which is an important issue in the theory of stochastic epidemic processes.  This enables us to provide rigorous convergence guarantees on the stochastic epidemic model itself, improving over the mean-field type convergence results of most prior work.
	There are two primary technical difficulties addressed in treating this problem: (i) constructing a means of tractably approximating the evolution of the process, so that the designed approximation is robust to the modeling error introduced by the applied moment closure, and (ii) guaranteeing that the designed controller causes the closed-loop system to drive the $SEIV$ process to extinction quickly.  As an application, we use the developed framework for optimizing the use of quarantines in containing an $SEIV$ epidemic outbreak.
\end{abstract}

\section{Introduction}

The modern study of epidemic models has been intense for the past several decades, with work dating back to the $1970$s \cite{Bailey1975}, and a slew of recent results from the control community coming in recent years \cite{Preciado2014,Nowzari2017,Ogura2016,Wang2014,Lee2016a,Watkins2015,Eshghi2016,Eshghi2017,Hota2017,Pare2017,Drakopoulos2014,Scaman2016,Watkins2017}.  Potential applications include message passing in complex wireless networks \cite{Wang2014}, competition between multiple mimetic behaviors in social networks \cite{Wei2013}, and the spread of biological disease \cite{Keeling2005,Eksin2017,Brauer2006,Yan2007}.  A recent review of the control of epidemic processes can be found in \cite{Nowzari2016}.

An important problem in the field is understanding how to efficiently control epidemic processes in such a manner so as to be able to provide rigorous performance guarantees on the statistics of the process.  This problem arises from the complexity inherent in networked epidemic systems: the stochastic dynamics which describe the fundamental aspects of the process entangle the components of the system's state, making their analysis inherently difficult. For sufficiently simple epidemics, such as the Susceptible-Infected-Susceptible $(SIS)$ process, using a mean-field type moment closure, in which the second-order moments of the system are approximated by products of first-order moments, yields dynamics which provide an upper-bound for the expectation of the stochastic process \cite{Simon2017}.  In general, this is not the case.  Indeed, even for simple models with multiple compartments, simulations have shown standard mean-field approximations to be unreliable proxies for the statistics of the underlying stochastic process (see, e.g., \cite{Watkins2016}).  As such, it remains a scientifically interesting question to develop control techniques for stochastic epidemic networked processes for which we can make rigorous claims about the statistics of the process.  As the types of controls available to authorities for the prevention of disease are often costly, understanding well how to optimize control resource use while still guaranteeing that the epidemic will end quickly is of critical importance in mitigating the effect of future epidemic threats, such as the predicted increase in endemic diseases due to climate change \cite{Dasaklis2012}, or the emergence of drug-resistant superbugs \cite{Metcalf2015a,Arias2009}. 

There are two approaches considered in prior work concerning the control of  stochastic epidemic models without mean-field approximation.  In one approach \cite{Scaman2016, Drakopoulos2014}, authors study policies which vary the healing rates of the nodes in the graph according to some computed priority order, wherein infected nodes with higher priority are treated before others.  These works guarantee that if sufficient healing resources are available, such a policy can control the $SIS$ epidemic to the disease-free state quickly.  In the other approach \cite{Watkins2017}, a model predictive controller is developed which controls the statistics of a discrete-time $SIS$ epidemic directly by varying the processes' spreading parameters.  As model predictive control is an important paradigm in control theory with many successful applications in diverse fields (see, e.g. \cite{Allgower2000,Ellis2014,Mesbah2016,Borelli2017} and the references therein), and it is unclear if the priority-order strategies of \cite{Scaman2016,Drakopoulos2014} can be extended to more general settings efficiently, we work on generalizing the approach of \cite{Watkins2017} here to the setting of continuous-time epidemics.

\emph{The primary contribution of this paper} is the development of a robust economic model predictive control scheme for the stochastic continuous-time $SEIV$ process which allocates control resources to the spreading network so as to minimize the resource consumption rate realized by the controller, while providing a strong theoretical convergence guarantee that the disease will be eliminated from the network quickly.  This is the first paper to use model predictive control in the context of epidemic containment for continuous-time epidemic processes, and one of the first papers to develop feedback controls for continuous-time stochastic epidemic processes.  With respect to our earlier work \cite{Watkins2017}, the text presented here differs in that it considers the control of continuous-time epidemic processes, wherein propagating the uncertainty of the system is inherently difficult, and rigorous arguments concerning the convergence of the closed-loop system are more difficult to develop.  The two key technical difficulties encountered in addressing this task are (i) providing a rigorous moment-closure approximation of the $SEIV$ process which provides bounds of its statistics, and (ii) providing a convergence analysis of the process when evolving under the designed model predictive controller which guarantees that the amount of time that passes until the process is disease-free is small.  

It is important to underline that we believe the closure techniques presented here will generalize to further epidemic models readily, as will the convergence arguments used.  We expect that the framework presented here will provide a fertile avenue for future research, providing a framework in which problems studied in prior works considering only the mean-field regime (such as \MFpapers) can be readily extended.

\paragraph*{\textbf{Organization of Remainder}}
The remainder of the paper is organized as follows.  In Section \ref{sec:problem}, we detail the $SEIV$ epidemic model, the control architecture, and the problem statement.  In Section \ref{sec:robust}, we develop a moment closure for the $SEIV$ epidemic model which provides the robust approximation guarantees which are required for the implementation of the desired control scheme.  In Section \ref{sec:control}, we perform a convergence analysis of the developed controller.  In Section \ref{sec:application}, we apply the developed controller to a problem of minimizing the number of nodes put into quarantine in order to guarantee a specified exponential decay of the number of exposed and infected nodes in the system.  In Section \ref{sec:summary}, we summarize the main results of the paper, and comment on avenues for future work. \oprocend
	
\paragraph*{\textbf{Notation and Terminology}}
We use $\realnonnegative$ to denote the set of non-negative real numbers, and $\Zplus$ denote the set of non-negative integers.  We denote by $[k]$ the set of the first $k$ positive integers, i.e. $[k] \triangleq \{1,2, \dots, k\},$ and by $[k]_0$ the first $k+1$ natural numbers, i.e. $[k]_0 \triangleq \{0,1,2,\dots,k\}.$  We let $e_i$ denote the $i$'th column of an identity matrix, with the appropriate dimension inferred from context.

We denote by $\E[X]$ the expectation of a random variable $X.$  Note that when the measure of the expectation is clear from context, we omit it.  When necessary, we explicitly include it as a subscript of the operator, i.e. $\E_{\mu}[X]$ is the expectation of $X$ with respect to the measure $\mu.$  When clear from context, we omit the initial condition $X(0)$ of a stochastic process.  \oprocend

\section{Model and Problem Statement}
\label{sec:problem}

In this section, we formally develop the model and the problem we study in this paper.  The particular construction of the $SEIV$ process we present here is our own, however the final system we arrive at is the same $SEIV$ model as studied in prior work (see, e.g., \cite{Nowzari2017,Hikal2014}).  Note that while epidemic models are inherently general mathematical abstractions which may be studied in a variety of contexts, the language we use throughout the paper is made specific to the context of biological epidemics for simplicity.

\subsection{Susceptible-Exposed-Infected-Vigilant Model} 
\label{subsec:SEIV}
We consider the dynamics of a general Susceptible-Exposed-Infected-Vigilant ($SEIV$) epidemic model.  In this model, each agent in a population is represented in a directed $n$-node spreading graph $\Graph = (\Nodes, \Edges)$ by a particular node $i \in \Nodes.$  At each time in the process, every node belongs to one of a set of the model's \emph{compartments}.  On an intuitive level, each compartment represents the stage of infection in which the agent currently resides.  In the $SEIV$ model, there are four types of model compartments: susceptible (denoted by the symbol $S$), exposed (denoted by the symbol $E$), infected (denoted by the symbol $I$), and vigilant (denoted by the symbol $V$).  When an agent is susceptible, we may think of it as healthy.  When an agent is exposed, we may think of it as having recently come into contact with a contagious disease, but without having yet outwardly displayed symptoms.   When an agent is infected, we may think of it as symptomatic.  When an agent is vigilant, we may think of it as actively protecting itself against exposure to the disease.

We denote by $X(t)$ a stochastic vector containing the compartmental memberships of each node at time $t.$  To make the notation as intuitive as possible, we index $X(t)$ in two dimensions: one which indicates the compartment which is being described, and the other the numerical label of the node.  As such, we denote by $X_i^C(t)$ an indicator random variable, taking the value $1$ if node $i$ is in compartment $C,$ where $C$ is one of the symbols $C \in \mathcal{L} \triangleq \{S,E,I,V\},$ and $0$ otherwise.  In this way, we see that for all times $t,$ we have that $\sum_{C \in \mathcal{L}} X_i^C(t) = 1$ for all $i,$ as each node belongs to precisely one compartment at all times.  We denote by $\StateSpace$ the set of all possible states of the $SEIV$ process.

\begin{figure}
	\centering
	\begin{tikzpicture}
	\node[circle,very thick,draw=black] (S) at (2,0.5) {$S$};
	\node[circle,very thick,draw=black] (E) at (4,0.5) {$E$};
	\node[circle,very thick,draw=black] (I) at (4,-0.5) {$I$};
	\node[circle,very thick,draw=black] (V) at (2,-0.5) {$V$};
	
	\path [->,very thick,dashed] (S) edge [bend left] node[xshift=0.0cm,yshift=0.25cm] {$\dQ_{4j} + \dR_{4j}$} (E);
	\path [->,very thick] (E) edge [bend left] node[xshift=0.35cm,yshift=0.0cm] {$\dS_4$} (I);
	\path[->,very thick] (I) edge [bend left] node[xshift=0.05cm, yshift=-0.25cm] {$\dU_4$} (V);
	\path[->,very thick] (V) edge [bend left] node[xshift=-0.35cm, yshift=0.0cm] {$\dP_4$} (S);
	\path[->,very thick] (S) edge [bend left] node[xshift=0.40cm, yshift=0.0cm] {$\dV_4$} (V);
	
	\node[circle,very thick,draw=black] (1) at (-2,0.5) {1};
	\node[circle,very thick,draw=black] (2) at (-1.4,-0.5) {2};
	\node[circle,very thick,draw=black] (6) at (-1,0.5) {3};
	\node[circle,very thick,draw=black] (7) at (0,0.5) {4};
	
	\path [->,very thick,color=black] (2) edge (1);
	\path [->,very thick,color=black] (1) edge (2);
	
	\path [->,very thick,color=black] (6) edge (1);
	\path [->,very thick,color=black] (7) edge (6);
	\path [->,very thick,color=black] (6) edge (7);
	\path [->,very thick,color=black] (6) edge (2);
	
	\node[] (anch1) at (1.7,1) {};
	\node[] (anch2) at (1.7,-.8) {};
	\node[] (ell1) at (1.2,0.5) {};
	\node[] (ell2) at (1.2,-0.5) {};
	
	\begin{pgfonlayer}{bg}
	\node[draw=black,fit={(S) (E) (I) (V) (anch1) (anch2)},thick ,inner sep=0.1ex,ellipse,fill=white] (I2circ) {};
	\node[circle,very thick,draw=black,opacity=0.45,fill=red] (S) at (2,0.5) {$S$};
	\node[circle,very thick,draw=black,opacity=0.45,fill=green] (E) at (4,0.5) {$E$};
	\node[circle,very thick,draw=black,opacity=0.45,fill=blue] (I) at (4,-0.5) {$I$};
	\node[circle,very thick,draw=black,opacity=0.45,fill=magenta] (V) at (2,-0.5) {$V$};
	
	\node[circle,very thick,draw=black,opacity=0.45,fill=red] (1) at (-2,0.5) {1};
	\node[circle,very thick,draw=black,opacity=0.45,fill=blue] (2) at (-1.4,-0.5) {2};
	\node[circle,very thick,draw=black,opacity=0.45,fill=magenta] (6) at (-1,0.5) {3};
	\node[circle,very thick,draw=black,opacity=0.45,fill=green] (7) at (0,0.5) {4};
	\end{pgfonlayer}
	\begin{pgfonlayer}{bbg}
	\path [very thick,] (7) edge (ell1);
	\path [very thick] (7) edge (ell2);
	\end{pgfonlayer}
	\end{tikzpicture}
	\caption{A compartmental diagram of the $SEIV$ process.  The transition process for node $4$ is explicitly illustrated, where the measures of the contact processes are included to indicate which process determines which transition, as described in Section \ref{subsec:SEIV}.}
	\label{fig:SEIV}
\end{figure}
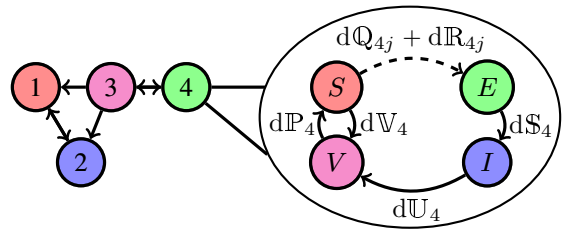
Figure \ref{fig:SEIV} provides an illustration of the $SEIV$ process, which can be posed as a system of \Ito integrals taken with respect to measures of independent Poisson processes as
\begin{equation}
\label{eq:SEIV_process}
\begin{aligned}
	& \deriv X_i^{S} = X_i^V \dP_i  - X_i^S \dV_i \\
	& \hspace{25 pt}  - \sum_{j \in \Neighbors_i}( X_i^S X_j^E \dQ_{ij} + X_i^S X_j^I \dR_{ij}),\\
	&\deriv X_i^{E} = \sum_{j \in \Neighbors_i} (X_i^S X_j^E \dQ_{ij} + X_i^S X_j^I \dR_{ij})  - X_i^E \dS_i,\\
	&\deriv X_i^{I} = X_i^E \dS_i - X_i^I \dU_i,\\
	&\deriv X_i^{V} = X_i^I \dU_i + X_i^S \dV_i - X_i^{V} \dP_i,
\end{aligned}
\end{equation}
where $\Neighbors_i$ is taken to be the set of in-neighbors of node $i$ on $\Graph,$ $\dP_i$ is the probability measure induced by the Poisson process generating transition events which take node $i$ from vigilant to susceptible, $\dQ_{ij}$ is the probability measure induced by the Poisson process generating transition events which take node $i$ from susceptible to exposed through contact with node $j$ when $j$ is exposed, $\dR_{ij}$ is the probability measure induced by the Poisson process generating transition events which take node $i$ from susceptible to exposed through contact with node $j$ when $j$ is infected, $\dS_{i}$ is the probability measure induced by the Poisson process which generates transition events which take node $i$ from exposed to infected, $\dU_{i}$ is the probability measure induced by the Poisson process generating transition events which transition node $i$ from infected to vigilant, and $\dV_i$ is the probability measure associated to the Poisson process which generates transitions from susceptible to vigilant.

We wish to control the process detailed by \eqref{eq:SEIV_process} through the dynamics of its expectation.  As such, we must formally derive the expectation dynamics of \eqref{eq:SEIV_process}, which after some technical arguments (see Appendix \ref{app:dynamics}) can be shown to be
\begin{equation}
	\label{eq:SEIV_diff}
	\begin{aligned}
	&\ddtop{\E[X_i^{S}]} = \alpha_i \E[X_i^V]  -  \xi_i\E[X_i^{S}]\\
	&\hspace{40 pt} - \sum_{j \in \Neighbors_i}(\beta_{ij} \E[X_i^S X_j^E] + \gamma_{ij} \E[X_i^S X_j^I]),\\
	&\ddtop{\E[X_i^{E}]} = \sum_{j \in \Neighbors_i} (\beta_{ij} \E[X_i^S X_j^E] + \gamma_{ij} \E[X_i^S X_j^I]) \\
	&\hspace{40 pt}  - \delta_{i} \E[X_i^E],\\
	&\ddtop{\E[X_i^{I}]} =  \delta_{i} \E[X_i^E] - \eta_i \E[X_i^I],\\
	&\ddtop{\E[X_i^{V}]} =  \eta_{i} \E[X_i^I] + \xi_i \E[X_i^{S}] - \alpha_i \E[X_i^{V}],
	\end{aligned}
\end{equation}
where $\alpha_i,$ $\beta_{ij},$ $\gamma_{ij},$ $\delta_{i},$ $\eta_i,$ and $\xi_i,$ are the rates of the processes associated to $\dP_i,$ $\dQ_{ij},$ $\dR_{ij},$ $\dS_i,$ $\dU_{i},$ and $\dV_i$ respectively, and are termed the \emph{spreading parameters}.  

The dynamics \eqref{eq:SEIV_diff} are a central object of study in this paper.  In particular, a major difficulty in typical epidemic control problems is in reconciling the fact that the induced first-moment dynamics \eqref{eq:SEIV_diff} are not closed, as they rely explicitly on the second-order moments $\E[X_i^S X_j^E]$ and $\E[X_i^S X_j^I],$ for which we do not have explicit dynamics.  Moreover, it can be shown that the dynamics of the second-order moments depend on third-order moments, and so on until the dynamics entail the expectation of all $4^n$ possible combinations of node-compartment states (see, e.g. the construction in \cite[Section 5]{Sahneh2013}).  Because of this, there is no known technique for propagating the exact expectations of the $SEIV$ process forward in time tractably, and techniques for closing \eqref{eq:SEIV_diff} by approximating the second-order moment terms - called \emph{moment closures} - are typically employed.  One contribution of this work is a novel moment closure which provides over- and under- approximations of compartmental membership probabilities for each node and every compartment, presented in Section \ref{sec:robust}.

\subsection{Control Architecture}
\label{subsec:control_arch}

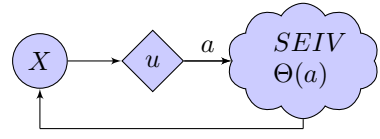
\begin{figure}
	\centering	
	\begin{tikzpicture}[auto, node distance=2cm,>=latex']
	\node [disk] (state) {$X$};
	\node [blkdiamond, right of=state, node distance= 1.5cm] (control) {$u$};
	\node [blkcloud, right of=control, node distance = 2.0cm, text width = 0.8cm] (network) {$SEIV$ $\Theta(\action)$};
	\coordinate [below of=network, node distance = 0.9cm] (blwnet) {};
	\coordinate [below of=state, node distance= 0.9cm] (blwstate) {};
	\draw [draw,->] (state) -- (control);
	\draw [draw,->] (control) -- (network);
	
	\path[->] (control) edge node {$\action$} (network);
	
	\draw [<-] (state) |- (blwstate);
	\draw [-] (network) |- (blwnet);
	\draw [-] (blwnet) -- (blwstate);
	\end{tikzpicture}
	\caption{A diagram of the control architecture studied in this paper.  The state $X$ of the process is observed by a controller $u,$ which then applies an action $a$ to the $SEIV$ process, which induces the set of spreading parameters $\Theta(\action)$ used to propagate the process forward.} \label{fig:sys_arch}
\end{figure}

We control the spreading network of the $SEIV$ process by way of applying control actions to the spreading graph of the process.  That is, our model assumes the existence of a set of control actions $\ActionSet,$ about which we assume that each action $\action$ fully determines the set of spreading parameters through a known map 
\begin{equation*}
	\Theta(\action) = \{\alpha(\action),\beta(\action),\gamma(\action),\delta(\action),\eta(\action),\xi(\action) \},
\end{equation*} 
and the cost of applying action $\action$ is given by $\Cost(\action),$ where $\Cost$ is a known function.  We study the case in which actions are applied to the network on some predefined set of times $\SamplingTimes \triangleq \{t \in \realnonnegative \, | \, t = \delt k, k \in \Zplus\}.$  The particular action applied is computed by the controller $u$ after observing the state of the system $X,$ where the computation performed is done so as to minimize the economic cost realized by the controller, while still providing some stability guarantee (see Section \ref{subsec:empc}).  Thus, our controller functions as a sampled-data model predictive controller for the continuous-time $SEIV$ process.  A diagram presenting this architecture is given in Figure \ref{fig:sys_arch}. 

Whenever the controller decides to change the applied action to $\action,$ the parameters of the process dynamics \eqref{eq:SEIV_diff} change to $\alpha_{i}(\action),$ $\beta_{ij}(\action),$ $\gamma_{ij}(\action),$ $\delta_{i}(\action),$ $\eta_i(\action),$ and $\xi_i(\action),$ respectively, where each value is assumed to be finite, is implicitly defined by $\Theta,$ and remains constant between update times.  Note that in the remainder of the paper, we will think of the application of a control action as inducing a probability measure which propagates the process forward; for convenience, we notate this measure $\Theta(\action),$ as it should cause no confusion.

In an application, the set $\ActionSet$ should be chosen so as to represent the types of actions a planner can take in order to effect the evolution of the epidemic, e.g. distributing medication, investing in awareness advertisement, or assigning people to quarantines.  The parameter map $\Theta$ should be defined to take each possible control action, and produce the particular set of spreading parameters induced by the action.  One should expect that giving a person medication should set her healing rate $\delta_i$ to a high value, whereas investing in advertisement will set the rate at which a node transitions to the vigilant state (i.e. $\xi_i$) to a high value, and quarantining a node will result in deleting some connections from the graph (i.e. setting $\beta_{ij}$ and $\gamma_{ij}$ to $0$) due to persons who ordinarily come into contact with each other not doing so any longer.
 
We assume that $\ActionSet$ is discrete and has at most finitely many elements, and that $\Cost(a)$ is finite for all $a \in \ActionSet.$  Note that these assumptions do not play a major role in the text - they do not influence our main results (Theorem \ref{thm:opt_apx} and \ref{thm:upper_bound}) whatsoever.  They make the problem setting more concrete, and allow us to define an appropriate optimization method easily, when examining the example application provided in Section \ref{sec:application}. 
	
Note that there are many works which have studied model predictive control with discrete control actions which precede this text \cite{Geyer2014,Preindl2013,Karamanakos2014,Richards2005,Rawlings2017a,Rawlings2018}.  As such, this feature should not be considered a primary contribution of the work.  In particular, the presence of discrete inputs does not materially change stability analysis significantly in many circumstances, as pointed out in \cite{Rawlings2017a}.  We believe extending the results in this paper to different action spaces to be straightforward.

\subsection{Economic Model Predictive Control}
\label{subsec:empc}

In contrast to traditional forms of model predictive control, economic model predictive control (EMPC) focuses on developing methods which optimize the economic performance of a system, while still guaranteeing underlying system properties such as stability (see \cite{Ellis2014,Ellis2016,Angeli2011,Amrit2011,Subramanian2014,Amrit2013,Angeli2012,Grune2017, Faulwasser2018, Rawlings2017b, Rawlings2018} for relevant background).  The controller we study in this paper is an economic model predictive controller which applies feasible, but possibly suboptimal, solutions to the optimization problem
\begin{equation}
	\label{prog:empc_def}
 	\underset{\action \in \ActionSet}{\text{minimize}} \{ \Cost(\action) \, | \,   \StabFun (X(t),\action) \leq 0\},
\end{equation}
at all times $t$ in the set of sampling times $\SamplingTimes \triangleq \{t \in \realnonnegative \, | \, t = \delt k, k \in \Zplus\}.$  The role of \eqref{prog:empc_def} in the closed-loop evolution of the controlled $SEIV$ process is in automatically generating actions from an implicit, nonlinear control law $\umpc$ which limits the rate at which economic costs are incurred by the applied actions while still guaranteeing stability.

Our role in designing the EMPC is in crafting an optimization problem of the form \eqref{prog:empc_def} such that applying actions which are feasible to \eqref{prog:empc_def} at all times $t$ in $\SamplingTimes,$ the closed-loop behavior of \eqref{eq:SEIV_diff} will drive the disease out of the network quickly - say, in an expected amount of time which grows linearly with the size of the initial infection.  To this end, we focus on the case in which the \emph{stability constraint function} $\StabFun (X(t), \action)$ takes the form
\begin{equation}
	\label{eq:Sdef}
	\StabFun(X(t),\action) \triangleq \E_{\Theta(\action)}[\ell(X(t+\delt)) - \ell(X(t)) \Exp{-r \delt} | X(t)],
\end{equation}
where $\ell$ is defined as the total number of exposed and infected nodes in the network, i.e. $\ell(X) \triangleq \sum_{i \in \Nodes} X_i^E + X_i^I,$ and $r$ is a chosen positive constant describing the desired decay rate of the closed-loop system.  The constraint $\StabFun (X(t),\action) \leq 0,$ which we refer to as the \emph{stability constraint}, is a common feature of nonlinear and economic model predictive controllers (see \cite[Section 3]{Mayne2000} and \cite[Section 3.3]{Ellis2014}, respectively), though the choice of stability constraint function varies depending on the context.  Ensuring that the stability constraint is satisfied whenever an action is updated plays a central role in our control by guaranteeing an appropriate notion of decay. 

While in full generality, it may be difficult to find a control action which is feasible to \eqref{prog:empc_def}, we assume that the controller has access to a stabilizing auxiliary control law $\uaux$ which provides such an action at any state.  Note that this assumption is common in the nonlinear and economic model predictive control literature (see \cite[Section 3]{Mayne2000} and \cite[Section 3.3]{Ellis2014}, respectively), due to the abject difficulty of solving \eqref{prog:empc_def} to global optimality caused by the nonconvexity induced by the nonlinearity of the dynamics.  Moreover, in the context of epidemic containment, we expect finding such controllers to be easy: if we are distributing medication, we may give medication to everyone, if we are distributing protective clothing, we may give protective clothing to everyone, if we are deciding who to quarantine, we may quarantine everyone.  In all such cases, if we do so, the epidemic will die out quickly.  This is to say that in the context of epidemic containment, we expect that the interesting problem is not generally in finding an abstract control law which will eliminate an epidemic quickly, it is in determining how to apply control actions efficiently, so as to use a nontrivial amount of economic resources while \emph{still} ensuring fast disease elimination.

\subsection{Problem Statement}
\label{subsec:problem_statement}

We address the problem of constructing a robust EMPC for the $SEIV$ process which guarantees that the epidemic attains membership in the set of disease free states quickly, while limiting the rate of resource consumption incurred by the actions applied by the controller.  Our work here is focused on addressing two key difficulties which arise from applying economic model predictive control to epidemic containment: (i) approximating the evolution of \eqref{eq:SEIV_diff} so as to have a robust approximation of the evolution of the processes' compartmental membership probabilities under a fixed control action $\action,$ and (ii) analyzing the convergence of the closed-loop process so as to provide a rigorous guarantee that the disease will be eliminated from the network quickly.

The difficulty of item (i) arises from the fact that in order to induce the desired decay property, the function $\StabFun (X(t),\action)$ must include information about the future expected state of the process, conditioned on the current state of the process and the action applied, i.e. a term of the form $\E_{\Theta(\action)}[\InfectExposed(X (t+ \delt))| X(t)],$ where $\InfectExposed$ is an appropriately chosen function.  Approximating such a function in a manner that guarantees that \eqref{prog:empc_def} is feasible requires that we construct a moment closure that can provide a rigorous upper-bound on $\E_{\Theta(\action)}[\InfectExposed(X (t+ \delt)) | X(t)],$ which is not a property guaranteed by any prior moment closure techniques.  We construct a novel moment closure with this property in Section \ref{sec:robust}.

The difficulty of item (ii) arises from the fact that the stability results which are typical to MPC literature prove asymptotic convergence to a connected, compact set which contains the origin (see, e.g. \cite{Allgower2000,Ellis2014,Mesbah2016,Borelli2017} and references therein).  Such a result is not useful in our context.  In general, converging to a set even at an exponential rate does \emph{not} guarantee that a given process \emph{ever} enters the set.  As the $SEIV$ process attains membership in the disease-free set in finite time almost surely under all but the most pathological control laws, it is primarily of interest to study how quickly it enters the set of disease free states, which motivates the novel convergence analysis presented in Section \ref{sec:control}.

\section{Robust Moment Closure for $SEIV$} \label{sec:robust}

In this section, we construct a robust moment closure approximation for the $SEIV$ process.  Essentially, this is a method for computing an outer approximation to the set of solutions of \eqref{eq:SEIV_diff}, which we cannot solve directly due to the second-order moments $\E[X_i^{C} X_j^{C^\prime}]$ not having known, analytic expressions.  The method we develop enables rigorously approximating conditional expectations of the form $\E_{\Theta(\action)}[\InfectExposed(X(t + \delt))|X(t)],$ which are necessary in order to enact the EMPC scheme detailed in Section \ref{subsec:empc}, as they are required to evaluate the feasibility of \eqref{prog:empc_def}.  Note that here, we only explicitly deal with approximating $\E_{\Theta(\action)}[\InfectExposed(X(t + \delt))|X(t)]$ as defined in Section \ref{subsec:empc}.  This is for brevity, as the techniques can be extended to other choices with more complicated analysis. 

The moment closure technique in this section is novel, and has the property that the compartmental membership probabilities of the system are bounded by known quantities for all time, in contrast to prior work.  In particular, most prior works in epidemic control literature use a mean-field type moment closure (see, e.g. \cite{Sahneh2013}), which replaces terms of the form $\E[X_i^C X_j^{C^\prime}]$ with the products $\E[X_i^C] \E[ X_j^{C^\prime}].$  While such an approximation may work well for sufficiently simple epidemic models \cite{Simon2017}, in general this approximation gives no rigorous accuracy guarantee, and there are known systems for which such moment closures result in poor approximations of the statistics of the underlying system \cite{Watkins2016,Cator2012}.  Note that nothing formal is known about the quality of mean-field approximations for the $SEIV$ process studied here. 

One may imagine that applying more sophisticated types of known moment closure techniques to epidemic processes, such as variants of derivative-matching methods \cite{Singh2006,Singh2011,Ghusinga2017,Soltani2014}, may avoid this issue.  However, this is not the case.  Such techniques give only weak guarantees of accuracy with respect to propagating uncertainty forward in time.  In particular, the dynamics resulting from such closures are guaranteed to be close to the true dynamics in a neighborhood of the initial condition used to generate the approximation.  Such a guarantee is not appropriate for our application, as we must be able to guarantee that our approximation of $\E_{\Theta(\action)}[\InfectExposed(X(t + \delt))|X(t)]$ is an upper-bound of its true value in order to asses the feasibility of \eqref{prog:empc_def} to guarantee stability.

\subsection{\Frechet Moment Closure of Moment Dynamics}
\label{subsec:frechet_apx}
We begin our formal developments by constructing a na{\"i}ve robust moment closure for the $SEIV$ expectation dynamics \eqref{eq:SEIV_diff}.  To present the developed system in a concise manner, we first introduce notation for the operators which characterize \Frechet inequalities.  These play a central role in our results.

\begin{lem}[\Frechet Inequalities \cite{Ruschendorf1991}]
	\label{lem:frechet}
	Define the operators $\FrechetLower(y,z) \triangleq \max\{0, y + z - 1\},$ and $\FrechetUpper(y,z) \triangleq \min\{y,z\}.$ 
	Let $\Pr$ be a probability measure on some event space $\Omega,$ and let $A$ and $B$ be events defined on $\Omega.$  Then, it holds that
	\begin{equation}
		\label{ineq:frechet}
		\FrechetLower(\Pr(A),\Pr(B)) \leq \Pr(A,B) \leq \FrechetUpper(\Pr(A),\Pr(B)),
	\end{equation}
	where the notation $\Pr(A,B)$ denotes the joint probability, i.e. the probability that the event $\{ A \cap B \}$ occurs.
\end{lem}
The role of the \Frechet inequalities in our development is that of bounding error in the dynamics due to the entanglement caused by the appearance of cross-product terms in \eqref{eq:SEIV_diff}.  Note that the \Frechet inequalities are the tightest inequalities estimating joint probabilities which are functions of only marginal probabilities, and are distribution-free \cite{Ruschendorf1991}.  Note also that our indicator random variables $X_i^C$ are Bernoulli random variables, and so the expectation of their product $\E[X_i^{C} X_j^{C^\prime}]$ can be written as the joint probability $\Pr(X_i^C = 1, X_j^{C^\prime} = 1),$ making the \Frechet inequalities an appropriate tool for our later analysis. As we cannot analytically propagate the distribution of \eqref{eq:SEIV_process} forward in time, and our dynamics \eqref{eq:SEIV_diff} only explicitly give us information about marginal probabilities, this is the best we can get for any given collection of marginals $\{\E[X_i^C]\}_{i \in \Nodes, C \in \mathcal{L}}.$

To arrive at a crude approximating system, we use the \Frechet bounds to over-approximate the dynamics on the states we designate as upper bounds, and under-approximate the dynamics on the states we designate as lower bounds.  The resulting approximating dynamics, along with their approximation guarantee is given in the following result.

\begin{theorem}[\Frechet Moment Closure for $SEIV$] \label{thm:frechet_apx}
	Let $\utilde{x}(0) = X(0) = \tilde{x}(0),$ and consider the solutions of the system of nonlinear ordinary differential equations
	\begin{equation}
		\label{eq:SEIV_frechet_apx_dyn}
		\begin{aligned}
		\dot{\tilde{x}}_i^S =& \alpha_i \tilde{x}_i^V - \xi_i \tilde{x}_i^S - \sum_{j \in \Neighbors_i} \beta_{ij} \FrechetLower(\tilde{x}_i^S,\utilde{x}_j^E) - \gamma_{ij} \FrechetLower(\tilde{x}_i^S, \utilde{x}_j^I),\\
		\dot{\utilde{x}}_i^S =& \alpha_i \utilde{x}_i^V - \xi_i \utilde{x}_i^S - \sum_{j \in \Neighbors_i} \beta_{ij} \FrechetUpper(\utilde{x}_i^S,\tilde{x}_j^E) -  \gamma_{ij} \FrechetUpper(\utilde{x}_i^S, \tilde{x}_j^I),\\
		\dot{\tilde{x}}_i^E =& \sum_{j \in \Neighbors_i}  \beta_{ij} \FrechetUpper(\tilde{x}_i^S,\tilde{x}_j^E) + \gamma_{ij} \FrechetUpper(\tilde{x}_i^S, \tilde{x}_j^I) - \delta_i\tilde{x}_i^E,\\
		\dot{\utilde{x}}_i^E =& \sum_{j \in \Neighbors_i} \beta_{ij} \FrechetLower(\utilde{x}_i^S,\utilde{x}_j^E) + \gamma_{ij} \FrechetLower(\utilde{x}_i^S, \utilde{x}_j^I) - \delta_i \utilde{x}_i^E,\\
		\dot{\tilde{x}}_i^I =&\delta_i \tilde{x}_i^E - \eta_i \tilde{x}_i^I,\\
		\dot{\utilde{x}}_i^I =&\delta_i\utilde{x}_i^E - \eta_i \utilde{x}_i^I,\\
		\dot{\tilde{x}}_i^V =& \eta_i \tilde{x}_i^I + \xi_i \tilde{x}_i^S - \alpha_i \tilde{x}_i^V,\\
		\dot{\utilde{x}}_i^V =& \eta_i \utilde{x}_i^I + \xi_i \utilde{x}_i^S - \alpha_i \utilde{x}_i^V.
		\end{aligned}
	\end{equation}
	 Then, for every compartment $C \in \mathcal{L}$ and each node $i \in \Nodes,$ the inclusion
		$\E[X_i^C(t) | X(0)] \in \left[\utilde{x}_i^C(t|0), \tilde{x}_i^C(t|0) \right],$
	holds for all $t \geq 0,$ where $\utilde{x}_i^C(t|0)$ and $\tilde{x}_i^C(t|0)$ are used to denote the under- and over-approximation of $\E[X_i^C(t)|X(0)]$ with respect to the dynamics \eqref{eq:SEIV_frechet_apx_dyn}, respectively.
\end{theorem}
\begin{proof}
We state and prove an extension to the comparison lemma (see, e.g., \cite[Lemma 3.4]{Khalil2002}) in Appendix \ref{app:lem:comparison}, which demonstrates a condition under which component-wise orderings are preserved under integration of the dynamics.  This is the central technical feature of the proof.  Having established an appropriate version of the comparison lemma, all that remains to prove the theorem is to prove that the dynamics \eqref{eq:SEIV_frechet_apx_dyn} satisfy the hypothesis of Lemma \ref{lem:comparison}, where the approximation dynamics \eqref{eq:SEIV_frechet_apx_dyn} are to be compared against the exact dynamics \eqref{eq:SEIV_diff}.  We show this in Appendix \ref{app:veracity}, completing the proof. \proofend
\end{proof}

While the approximation dynamics given by \eqref{eq:SEIV_frechet_apx_dyn} are technically correct and \emph{can} be used to estimate $\E_{\Theta(\action)}[\InfectExposed(X(t+\delt)) | X(t)]$ rigorously, they are not without fault.  In particular, we may note that they include no mechanism for ensuring that the approximations remain bounded in the unit interval.  This means that there is no way of ruling out the possibility that the solutions of the approximating dynamics \eqref{eq:SEIV_frechet_apx_dyn} become trivial at some time.  Indeed, as demonstrated in Figure \ref{fig:frechet_apx}, it is the case the approximations generated from this system give upper-bounds which exceed one eventually, and hence become trivial.  This occurs because the error which enters the approximating system is integrated through time into the evolution of the approximations of the marginal probabilities.  Addressing this issue so as to guarantee that the approximations always remain nontrivial requires a deeper analysis, which comes in the following subsection.

\begin{figure}
	\centering
	\includegraphics[width = 0.5\textwidth]{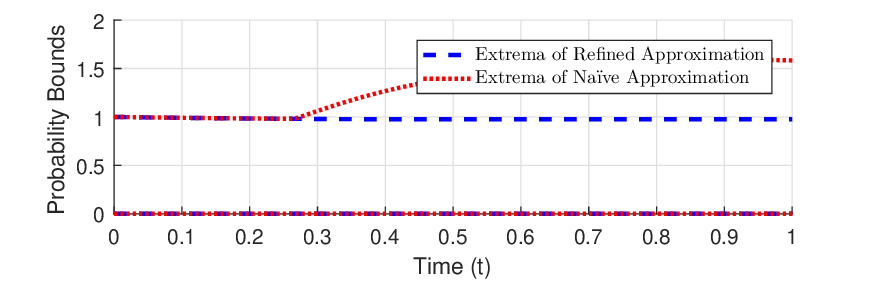}
	\caption{A demonstration that the crude \Frechet approximation system given by \eqref{eq:f_robust_dyn} does not generate approximations such that the estimates of compartmental membership probabilities remain bounded in the unit interval, whereas the refined approximation dynamics \eqref{eq:refined_dyn} does.  This underscores the fact that when designing a moment closure approximation, care must be taken to ensure that the solutions of the resulting approximation system behave reasonably.} \label{fig:frechet_apx}
\end{figure}

\subsection{Refining the Robust \Frechet Moment Closure}
\label{subsec:refining}

We begin this subsection by introducing a proposition which equips us with a formal test for determining when one set of approximation dynamics is better than another, in terms of set inclusion.  The result is stated as follows:

\begin{prop}[Constructing Tighter Approximations] 
\label{prop:tighter}
Consider the functions $\utilde{f}$ and $\tilde{f}$ defining the dynamics \eqref{eq:SEIV_frechet_apx_dyn}, and let $\{\utilde{x}, \tilde{x}\}$ be the solutions of \eqref{eq:SEIV_frechet_apx_dyn}.  Suppose $\ubar{f}$ and $\bar{f}$ are Lipschitz continuous functions on $\real^{(\Nodes \times \mathcal{L}) \times 2}$ which satisfy
\begin{equation}
	[\ubar{f}_i^C(\ubar{x},\bar{x}), \bar{f}_i^C(\ubar{x},\bar{x})] \subseteq [\utilde{f}_i^C(\utilde{x},\tilde{x}),\tilde{f}_i^C(\utilde{x},\tilde{x})]
\end{equation}
on $\bar{\mathcal{X}}_i^C \cup \ubar{\mathcal{X}}_i^C$ for all $(i,C) \in \Nodes \times \mathcal{L},$ where $\bar{\mathcal{X}}_i^C$ is the subset of points $(\ubar{x}, \bar{x},\utilde{x},\tilde{x})$ of $\real^{(\Nodes \times \mathcal{L}) \times 4}$ such that $\bar{x}_i^C = \tilde{x}_i^C,$ and $[\ubar{x}_j^{C^\prime},\bar{x}_j^{C^\prime}] \subseteq [\utilde{x}_j^{C^\prime}, \tilde{x}_j^{C^\prime}]$ holds for all $(j,C^\prime) \in \Nodes \times \mathcal{L},$ and $\ubar{\mathcal{X}}_i^C$ is defined similarly.  Then, the system
\begin{equation}
\begin{aligned}
\dot{\bar{x}} = \bar{f}(\ubar{x},\bar{x}),\\
\dot{\ubar{x}} = \ubar{f}(\ubar{x},\bar{x}),\\
\end{aligned}
\end{equation}
with initial conditions $\ubar{x}(0) = \utilde{x}(0),$ $\bar{x}(0) = \tilde{x}(0)$ has solutions $\{\ubar{x},\bar{x}\}$ which satisfy $[\ubar{x}_i^C(t),\bar{x}_i^C(t)] \subseteq [\utilde{x}_i^C(t),\tilde{x}_i^C(t)]$
for all $(i,C) \in \Nodes \times \mathcal{L}$ and any time $t \geq 0.$ 
\end{prop}

\begin{proof}
	See Appendix \ref{app:prop:tighter}. \proofend
\end{proof}

\noindent Proposition \ref{prop:tighter} gives a means for testing whether or not the crude approximation given by \eqref{eq:SEIV_frechet_apx_dyn} is improved upon by a new candidate approximation, and is the main result we use in developing a refinement of \eqref{eq:SEIV_frechet_apx_dyn}.  

The most significant modification over \eqref{eq:SEIV_frechet_apx_dyn} required to ensure all estimates remain bounded to the unit interval for all time is in ensuring the dynamics of each approximating upper-bound are non-increasing when the approximating upper-bound is equal to one, and the dynamics of the approximating lower-bound are non-decreasing when it is equal to zero.  To accomplish such an approximation, we introduce complement bounds to the dynamics, which we define formally as follows.

\begin{defn}[Complement Bounding Operator]
	{\rm
	Let $C \in \{S,E,I,V\}$ be a compartmental label of the $SEIV$ spreading process.  We define the complement upper-bounding operator associated to $C$ as the nonlinear operator
	\begin{equation}
		\Boundupri{C}y \triangleq \min\{1 - \bar{x}_i^C, y\}.
	\end{equation}
	}
\end{defn}

In essence, we see that improvements over \eqref{eq:SEIV_frechet_apx_dyn} can be made if we replace all instances of variables which can cause unbounded growth with appropriate complement bounds.  By doing so systematically, we arrive at a better approximating system, which we give here in Theorem \ref{thm:refined}.

\begin{theorem}[Refined \Frechet Moment Closure  of $SEIV$]
	\label{thm:refined}
	Let $\ubar{x}(0) = X(0) = \bar{x}(0),$ and consider the solutions of the system of nonlinear ordinary differential equations
	\begin{equation}
	\label{eq:refined_dyn}
	\begin{aligned}
		\dot{\bar{x}}_i^S =& \alpha_i \Boundupri{S}\bar{x}_i^V - \xi_i \bar{x}_i^S - \sum_{j \in \Neighbors_i} \beta_{ij} \FrechetLower(\bar{x}_i^S,\ubar{x}_j^E) - \gamma_{ij} \FrechetLower(\bar{x}_i^S, \ubar{x}_j^I),\\
		\dot{\ubar{x}}_i^S =& \alpha_i \ubar{x}_i^V - \xi_i \ubar{x}_i^S- \sum_{j \in \Neighbors_i} \beta_{ij}\FrechetUpper(\ubar{x}_i^S,\bar{x}_j^E) - \gamma_{ij} \FrechetUpper(\ubar{x}_i^S, \bar{x}_j^I),\\
		\dot{\bar{x}}_i^E =& \sum_{j \in \Neighbors_i} \beta_{ij} \FrechetUpper(\Boundupri{E}\bar{x}_i^S,\bar{x}_j^E) + \gamma_{ij}\FrechetUpper(\Boundupri{E} \bar{x}_i^S, \bar{x}_j^I) -  \delta_i \bar{x}_i^E,\\
		\dot{\ubar{x}}_i^E =& \sum_{j \in \Neighbors_i} \beta_{ij} \FrechetLower(\ubar{x}_i^S,\ubar{x}_j^E) + \gamma_{ij}\FrechetLower(\ubar{x}_i^S, \ubar{x}_j^I) - \delta_i \ubar{x}_i^E,\\
		\dot{\bar{x}}_i^I =& \delta_i \Boundupri{I}\bar{x}_i^E - \eta_i \bar{x}_i^I,\\
		\dot{\ubar{x}}_i^I =& \delta_i \ubar{x}_i^E - \eta_i \ubar{x}_i^I,\\
		\dot{\bar{x}}_i^V =& \eta_i \Boundupri{V}\bar{x}_i^I + \xi_i \Boundupri{V}\bar{x}_i^S - \alpha_i\bar{x}_i^V,\\
		\dot{\ubar{x}}_i^V =&  \eta_i \ubar{x}_i^I +  \xi_i \ubar{x}_i^S - \alpha_i \ubar{x}_i^V
	\end{aligned}
	\end{equation}
	Then, for every compartment $C \in \mathcal{L}$ and each node $i \in \Nodes,$
	\begin{equation}
		\label{incl:refined_frechet_apx}
		\E[X_i^C(t) | X(0)] \in \left[\ubar{x}_i^C(t|0), \bar{x}_i^C(t|0) \right] \subseteq [0,1],\\
	\end{equation}
	and
	\begin{equation}
		\label{incl:inner_outer}
		\left[\ubar{x}_i^C(t|0), \bar{x}_i^C(t|0) \right] \subseteq \left[\utilde{x}_i^C(t|0), \tilde{x}_i^C(t|0) \right],
	\end{equation}
	hold for all $t \geq 0,$ where $\ubar{x}_i^C(t|0)$ and $\bar{x}_i^C(t|0)$ are used to denote the under- and over-approximation of $\E[X_i^C(t)|X(0)]$ with respect to the dynamics \eqref{eq:refined_dyn}, and $\utilde{x}_i^C(t|0)$ and $\tilde{x}_i^C(t|0)$ are used to denote the under- and over-approximation of $\E[X_i^C(t)|X(0)]$ with respect to the dynamics \eqref{eq:SEIV_frechet_apx_dyn}.
\end{theorem}

\begin{proof}
	See Appendix \ref{app:thm:refined}. \proofend
\end{proof}

We can think of the process that we have used to arrive at \eqref{eq:refined_dyn} from \eqref{eq:SEIV_diff} as one of successively pruning the set of trajectories permitted by the approximating systems.  In the step where we moved from \eqref{eq:SEIV_diff} to \eqref{eq:SEIV_frechet_apx_dyn}, we used the \Frechet probability bounds to constrain the set of trajectories our system may admit as solutions to a superset of the set of probability measures the \Frechet approximations permit.  In moving from \eqref{eq:SEIV_frechet_apx_dyn} to \eqref{eq:refined_dyn}, we further restrict the set of solutions to those which satisfy simple complementarity bounds.  In so doing, each step improved the accuracy with which the dynamics are approximated.  We now consider how to use \eqref{eq:refined_dyn} in order to approximate the term $E_{\Theta(\action)}[\InfectExposed(X(t+\delt)) | X(t)]$ in the constraint function.

\subsection{Decay Constraint Approximation}
\label{subsec:decay_apx}

It is obvious that if we set $\ubar{x}(0) = X(0) = \bar{x}(0)$ and integrate \eqref{eq:refined_dyn} over the interval $[0,t],$ we get that
\begin{equation}
	\label{ineq:apx_exp}
	 \E_{\Theta(\action)}[\InfectExposed(X(t)) | X(0)] \leq \InfectExposed(\bar{x}(t))
\end{equation}
holds.  However, \eqref{ineq:apx_exp} is not the best approximation that can be derived from the solutions of \eqref{eq:refined_dyn}.  The optimal approximation which can be obtained from propagating the dynamics \eqref{eq:refined_dyn} can be found efficiently via linear programming, as we state formally in the following result and its proof.

\begin{theorem}[Optimal Approximation of Decay Function]
	\label{thm:opt_apx}
	Let $\ubar{x}(0) = X(0) = \bar{x}(0),$ and consider the solutions $(\ubar{x},\bar{x})$ of \eqref{eq:refined_dyn} evaluated at time $t.$  It holds that
	\begin{equation}
	\label{ineq:opt_bound}
	\begin{aligned}
	& \E_{\Theta(\action)}[\InfectExposed(X(t)) | X(0)] \leq \psi(X(0),a) \triangleq \\
	& \hspace{5 pt} \sum_{i \in \Nodes} \min\{ \bar{x}_i^E(t|0) + \bar{x}_i^I(t|0), 1 - \ubar{x}_i^S(t|0) - \ubar{x}_i^V(t|0) \},
	\end{aligned}
	\end{equation} where the bound is the tightest which can be derived from the inclusions generated by integrating \eqref{eq:refined_dyn}, and in particular is pointwise tighter than the bound \eqref{ineq:apx_exp}.
\end{theorem}

\begin{proof}
See Appendix \ref{app:thm:opt_apx}. \proofend
\end{proof}

From Theorem \ref{thm:opt_apx}, we can see that while the states of the dynamical system \eqref{eq:refined_dyn} are not themselves the best approximation possible for the conditional expectation $\E_{\Theta(\action)}[\InfectExposed(X(t)) | X(0)],$ the optimal approximation can be defined as a nonlinear output function $\sum_{i \in \Nodes} \min\{ \bar{x}_i^E(t|0) + \bar{x}_i^I(t|0), 1 - \ubar{x}_i^S(t|0) - \ubar{x}_i^V(t|0) \},$ and can thus be computed as efficiently as integrating the dynamics \eqref{eq:refined_dyn}.  Since we may repeat the arguments used in support of Theorem \ref{thm:opt_apx} just as easily with the initial time taking the value $t,$ it is clear that we can efficiently compute an upper-bound on $\E_{\Theta(\action)}[\ell(X(t + \delt) )| X(t)]$ from any state, at any time.  We use this approximation scheme in a simulated application in Section \ref{sec:application}, where we see (Figure \ref{fig:sim_results}) that the approximation generates nontrivial approximations that are useful for controlling the $SEIV$ process.

\section{Convergence Analysis}
\label{sec:control}

In this section, we analyze the convergence of the controller induced by applying suboptimal solutions of \eqref{prog:empc_def}, where we choose $\StabFun (X(t),\action)$ as in \eqref{eq:Sdef}, and $\E_{\Theta(\action)} [\ell(X(t+\delt)) | X(t)]$ is approximated by the method described in Theorem \ref{thm:opt_apx} (see Section \ref{subsec:decay_apx}).  Abstractly, we are most concerned with ensuring that the epidemic attains membership in the set of disease-free states quickly.  To formalize this notion, we define the \emph{elimination time} of the process as follows:
\begin{defn}[Elimination Time]
	{\rm
		The elimination time $\telim$ of an $SEIV$ epidemic is the first time at which all nodes are in neither the exposed nor the infected compartment, i.e. $\telim \triangleq \inf\{t \geq 0 \, | \, \InfectExposed(X(t)) = 0\},$ where $\InfectExposed(X(t)) \triangleq \sum_{i \in \Nodes} X_i^E(t) + X_i^I(t).$
	}
\end{defn}
We seek a guarantee on the expected elimination time of the $SEIV$ process under the designed EMPC scheme.  Since the $SEIV$ process, as well as most compartmental epidemic processes in general, attains the disease-free set in finite time almost surely under all but the most pathological control laws, we must consider a strong notion of stability for it to be meaningful.  Here, we show that the expectation $\telim$ grows slowly with the size of the initial infection.

The analysis that we perform to arrive at an upper bound for the expectation of the elimination time $\telim$ relies critically on knowledge of the evolution of the expected number of exposed and infected nodes in the graph.  As such, we first analyze this expectation (in Section \ref{subsec:EI_bound}), and then analyze the expectation of the elimination time (in Section \ref{subsec:elim_bound}). 

\subsection{\small Bounding the Expected Number of Exposed and Infected Nodes}
\label{subsec:EI_bound}
We find that the proposed EMPC method uniformly exponentially eliminates the epidemic in expectation, with respect to the set of sampling times $\SamplingTimes,$ stated formally as follows. 

\begin{theorem}[Exponential Elimination Under EMPC]
	\label{thm:online}
	Let $X(0)$ be the initial state of the $SEIV$ process, choose any $r > 0 ,$ and any $\delt > 0.$  Suppose an auxiliary control policy $\uaux$ exists such that 
	\begin{equation}
	\label{ineq:e_decay}
	\begin{aligned}
	& \E_{\uaux}[\InfectExposed(X(t + \delt))| X(t)] \leq \InfectExposed(X(t)) \Exp{-r \delt},
	\end{aligned}
	\end{equation}
	for all $t \in \SamplingTimes.$  Then, the evolution of the $SEIV$ process under the policy $\umpc$ generated by the economic model predictive controller specified in Section \ref{subsec:empc} satisfies
	\begin{equation}
	\label{ineq:alg1_decay_samples}
	\begin{aligned}
	& \E_{\umpc}[\InfectExposed(X(t)) | X(0)] \leq \InfectExposed(X(0)) \Exp{-r t},
	\end{aligned}
	\end{equation}
for all $t \in \SamplingTimes.$  Moreover, the bound \eqref{ineq:alg1_decay_samples} is tight.
\end{theorem}

\begin{proof}
	See Appendix \ref{app:thm:online}. \proofend
\end{proof}

The proof of Theorem \ref{thm:online} follows from an induction argument, which makes appeals to the expectation decay constraint \eqref{ineq:e_decay}, fundamental tools from the theory of probability, and the decay property encoded in the actions contained in feasible set of \eqref{prog:empc_def} through the constraint $\StabFun (X(t),\action) \leq 0.$  With respect to our problem, the principle importance of Theorem \ref{thm:online} is in allowing us to rigorously analyze the elimination time of the process, which we perform in the following subsection.

\subsection{Bounding the Expected Elimination Time}
\label{subsec:elim_bound}

While it is intuitive that an $SEIV$ process in which the total count of exposed and infected nodes decays exponentially quickly might have a small elimination time, there is no immediately apparent link between the two concepts.  Moreover, the decay property guaranteed by Theorem \ref{thm:online} is \emph{not} uniform exponential elimination; the exponential decay is only guaranteed on a countable subset of times.  As such, we may only infer anything about the expected number of exposed and infected nodes on a small subset of times, and must use this information to prove a result on the elimination time of the process.  We know of no previously published technique for doing so in the literature, but fortunately it can be done, as shown in the following result and its proof:
\begin{theorem}[Bound on Expected Elimination Time]
	\label{thm:upper_bound}
	 Suppose the $SEIV$ process evolves under the policy $\umpc$ generated by the EMPC method detailed in Section \ref{subsec:empc}, where $\StabFun (X(t),\action)$ is defined by \eqref{eq:Sdef}, and $\E_{\Theta(\action)}[\ell(X(t + \delt)) | X(t)]$ is approximated by the method detailed in Theorem \ref{thm:opt_apx}.  Then, the expected elimination time satisfies
		\begin{equation}
		\label{ineq:elim_bound}
		\E_{\umpc}[\telim | X(0)] \leq \tone + \frac{\me^{-r \tone}}{1 - \me^{-r \delt }} \delt \InfectExposed(X(0)),
		\end{equation}
	where $\tone$ is the first time in the sampling time set such that the expected number of exposed and infected nodes is less than one, which can be shown to be
		\begin{equation}
			\label{eq:time_to_one}
			\tau_1 = \left \lceil \frac{\log(\InfectExposed(X(0)))}{r \delt} \right \rceil \delt,
		\end{equation}
		where $\lceil a \rceil$ denotes the smallest integer larger than $a,$ i.e. the ceil of $a.$
		Moreover, the bound \eqref{ineq:elim_bound} is tight.
\end{theorem}
\begin{proof}
The essence of this argument relies on approximating the integral which defines the expected elimination time.  The approximation occurs in three steps:(i) representing the expectation as an integral of the distribution function of the elimination time random variable, (ii) finding a convergent, closed-form approximation to the distribution function of the elimination time random variable, and (iii) evaluating the integral of the approximated distribution function.

Since $\telim$ is a non-negative random variable which takes values on the real line, we may use the layer-cake representation of expectation to identify the equivalence
\begin{equation}
\label{eq:layer_cake}
\E_{\umpc}[\telim | X(0)] \triangleq \int_{0}^{\infty} 1 - F_{\telim}(h) d h
\end{equation}
where $F_{\telim}$ is the distribution function of $\telim,$ i.e. $F_{\telim} (t) \triangleq \Pr(\telim \leq t).$
Since once $\InfectExposed(X(\tau)) = 0$ for some $\tau,$ it holds for all $t \geq \tau,$ we have the identity $\Pr(\telim \leq t) = \Pr(\InfectExposed(X(t)) = 0),$
which is equivalent to the expression $1 - \Pr(\telim \leq t) = \Pr(\InfectExposed(X(t)) > 0).$

As such, we may construct an upper-bound on $1 - F_{\telim}(t)$ by constructing an upper-bound on $\Pr(\InfectExposed(X(t)) > 0).$  We do so using the decay properties already proven of the designed controller.  In particular, Theorem \ref{thm:online} gives that
\begin{equation}
	\label{ineq:exp_con_def}
	\E_{\umpc}[\InfectExposed(X(t)) | X(0)] \leq \InfectExposed(X(0)) \Exp{-r t}
\end{equation}
holds for all $t \in \SamplingTimes,$ and so we may upper-bound $\Pr(\InfectExposed(X(t)) > 0)$ for all $t \in \SamplingTimes$ by the optimal value of
\begin{subequations}
\label{prog:var_bound}
\begin{align}
& \underset{p \in \Dists_{[n]_0}}{\text{maximize}}
& & \sum_{k = 1}^{n} p_{k} \\
& \text{subject to}
& & \sum_{k = 0}^{n} k p_{k} \leq \InfectExposed(X(0)) \Exp{-r \left \lfloor \frac{t}{\delt} \right \rfloor \delt}, \label{ineq:exp_con}
\end{align}
\end{subequations}
where $\lfloor a \rfloor$ denotes the largest integer smaller than $a,$ and $\Dists_{[n]_0}$ is the set of all possible marginal probability assignments over $[n]_0,$ so chosen because the number of exposed and infected nodes in the graph must take value on $[n]_0,$ and the constraint \eqref{ineq:exp_con} enforces the expectation inequality \eqref{ineq:exp_con_def}.  Since the left-hand-side of inequality \eqref{ineq:exp_con} is least sensitive to increases in $p_1$ for any value $p_k$ with $k \geq 1,$ one may show that the optimal value of \eqref{prog:var_bound} can be computed as $\min\{1, \InfectExposed(X(0)) \Exp{-r \left \lfloor \frac{t}{\delt} \right \rfloor \delt} \}.$  Hence,
\begin{equation}
	\label{ineq:pr_ineq}
	\Pr(\InfectExposed(X(t)) > 0 | X(0)) \leq \min\{1, \InfectExposed(X(0)) \me^{-r \lfloor \frac{t}{\delt} \rfloor \delt} \}
\end{equation}
holds for all $t \in \SamplingTimes,$ where the bound \eqref{ineq:pr_ineq} is tight.  

Since the distribution function $F_{\telim}$ is non-decreasing with respect to $t$ by definition, and $1 - F_{\telim}(t) = \Pr(\InfectExposed(X(t)) > 0),$ we have that $\Pr(\InfectExposed(X(t)) > 0)$ is \emph{non-increasing} with respect to $t.$  Hence \eqref{ineq:pr_ineq} holds for \emph{all} times $t \geq 0.$ Thus, \eqref{eq:layer_cake} and \eqref{ineq:pr_ineq} together imply the inequality
\begin{equation}
\label{ineq:exp_bound}
\E_{\umpc}[\telim | X(0)] \leq \int_{0}^{\infty} \min\left\{1, \InfectExposed(X(0)) \me^{-r \left\lfloor \frac{h}{\delt} \right\rfloor \delt} \right\} d h.
\end{equation}	

We now seek a closed-form expression of the right hand side of \eqref{ineq:exp_bound}.  Defining $\tone$ as 
$$\tone \triangleq \inf\{t \in \realnonnegative \, |\,  \InfectExposed(X(0)) \me^{-r \left\lfloor \frac{t}{\delt} \right\rfloor \delt} \leq 1 \},$$
we can evaluate $\tone$ to satisfy the claimed identity \eqref{eq:time_to_one}, and can then rewrite \eqref{ineq:exp_bound} as
\begin{equation}
\label{ineq:elim_time_bound}
\E_{\umpc}[\telim | X(0)] \leq \tone + \sum_{j = \frac{\tone}{\delt} }^{\infty} \InfectExposed(X(0)) \delt \left( \me^{-r \delt} \right)^j.
\end{equation}	
Evaluating the tail of the geometric sum in \eqref{ineq:elim_time_bound} gives
\begin{equation*}
\E_{\umpc}[\telim | X(0)] \leq \tone + \frac{\me^{-r \tone}}{1 - \me^{-r \delt }} \delt \InfectExposed(X(0)),
\end{equation*}	
which is as stated in the theorem's hypothesis.

Finally, note that since the bound on the distribution function $F_{\telim}$ derived is optimal at all times among all such bounds which use only the bound on the expected number of exposed and infected nodes in the graph provided by Theorem \ref{thm:online}, and the bound provided by Theorem \ref{thm:online} is itself tight, it follows as well that the upper bound \eqref{ineq:exp_bound} is optimal among all such guarantees that can be provided by the designed EMPC.  This completes the proof.
\proofend

\end{proof}

Since \eqref{ineq:exp_bound} grows at worst linearly with respect to the number of initially infected and exposed nodes in the graph, the derived bound certifies that the designed EMPC scheme eliminates the epidemic from the network quickly.  Thus, the problem stated in Section \ref{subsec:problem_statement} has been appropriately solved.

Note that the proof given for Theorem \ref{thm:upper_bound} relies critically on the fact that the expected number of infected and exposed nodes decays exponentially quickly, as if this were not the case, the approximation used for the integrand would not be integrable, and the resulting approximation would be trivial.  Moreover, the argument relies on the topology of the state-space of epidemic processes in order to guarantee that the controller attains membership in the targeted set of states quickly.  Indeed, if it were not the case that the optimal value of \eqref{prog:var_bound} is bounded away from one after only a short amount of time, our attempt at approximating the expectation of $\telim$ meaningfully would fail as well.  Since this occurs \emph{only} because $\ell(X)$ \emph{must} take values on the set $[n]_0,$ it follows that such a convergence argument will not generalize to EMPC schemes on general state spaces, but \emph{will} generalize to other epidemic process readily, as all such processes taking place on finite graphs evolve on finite state spaces.

\begin{remark}[Use of Alternate Approximation Methods]
	{\rm
	As noted in Section \ref{sec:robust}, the motivation for constructing a robust moment closure as we have is to be certain that the convergence guarantees we arrive at give us information on the behavior of the statistics of the $SEIV$ process.  By adequately accounting for the worst-case introduction of approximation errors in the dynamics as we have, this goal was accomplished.  However, one may wish to use alternate approximations of the dynamics in regions of the state space wherein robustness is unimportant, in order to attempt to find actions which improve cost performance.
	
	For example, if for an $n$ node graph, we partition $\StateSpace$ into sets $\mathcal{W}_k = \{Y \in \StateSpace \, | \, \ell(Y) = k\}$ for each $k \in [n]_0,$ we may decide to use a mean-field type moment closure or a Monte Carlo simulation to approximate $\StabFun (X(t),\action)$ when $X \in \cup_{k = 0}^{\tilde{k}} \mathcal{W}_k$ for some $\tilde{k} < n,$ and use the robust approximations derived in Section \ref{sec:robust} otherwise.  In this situation, the controller developed here will drive $X$ into $\cup_{k = 0}^{\tilde{k}} \mathcal{W}_k$ quickly whenever it leaves the set.  That is, such a controller will be guaranteed to keep the total number of exposed and infected nodes below $\tilde{k}$ efficiently, and will otherwise use other approximations in order to attempt to use fewer resources.  The math required to formalize and prove the above claim is very similar to that which was used to prove the convergence guarantees in this section, and as such we will not present it here. \oprocend
	}
\end{remark}

\section{Application: Optimizing Quarantine Use} 
\label{sec:application}
In this section, we present a concrete application for the developed EMPC framework.  Note that while we only present one application here, the general principles contained in Sections \ref{sec:robust} and \ref{sec:control} are not constrained to this context.  Indeed, whenever anyone should want to consider a new application, all one needs to do is specify an action space $\ActionSet,$ a parameter map $\Theta,$ an auxiliary control law $\uaux,$ and an appropriate optimization method.

\subsection{Quarantine Model for $SEIV$}
\label{subsec:quarantine_model}

We consider the problem of strategically removing nodes from the spreading graph in order to efficiently drive an $SEIV$ epidemic to extinction quickly.  This is a mathematical model for the practical problem of deciding who to quarantine, and for how long, in the presence of an epidemic contagion.  Note that in the context of this problem, the maximum realized resource use rate is the number of quarantine beds needed throughout the course of the epidemic, and as such provides a reasonable index for evaluating the cost of the controller.  Furthermore, given that world governments are currently in the process of providing disease control agencies with sweeping authority to quarantine individuals exposed to infectious disease (see, e.g., the recent U.S. bill \cite{CDC2017}) despite recommendations from the medical community (see, e.g., \cite{Drazen2014}), understanding the mathematics of when quarantining is necessary for the control of a disease is of utmost importance.

We represent control actions here by an $n$-dimensional vector $\action$ in which $\action_i = 1$ if and only if the $i$'th node is removed from the spreading graph (i.e., quarantined), and $\action_i = 0$ otherwise. With this notation, we may represent our action space as $\ActionSet = \{0,1\}^n.$  We model quarantining a node by removing its outgoing edges from the spreading graph, i.e. we have for all pairs $(i,j),$ the exposure rates have the functional forms $\beta_{ij}(\action_j) \triangleq \BetaIJmax - \BetaIJmax \action_j,$ and $\gamma_{ij}(\action_j) \triangleq \GammaIJmax - \GammaIJmax \action_j,$
where each $\action_j$ is restricted to the set $\{0,1\}.$  For simplicity, we assume that cost of quarantining nodes is additive, and so may be represented as $\Cost(\action) \triangleq \sum_{i \in \Nodes} \action_i.$
In this context, the value of the cost function evaluated for a particular control action $\action$ is representative of the number of beds required to implement the quarantine strategy.  Applying economic model predictive control to this problem explicitly attempts to minimize the number of beds used in execution.

Note that our work here is \emph{not} the first to study the problem of quarantine management for models of biological disease.  Typical works from the pre-existing literature study quarantine management problems for mean-field epidemic models, and model a node being in quarantine by adding an additional compartment to the compartmental spreading model, which does not interact with any other compartments (see, e.g., \cite{Brauer2006,Yan2007} for specific instances).  The control design is done by way of varying the rate at which nodes transition to the quarantine compartment, with actuating the rate coming at a given cost.  Our model here is similar, in that our quarantined nodes do not interact with the rest of the network, and placing the node in quarantine comes at a cost to the controller.  Note, however, that our work here is the first work which considers quarantine optimization for stochastic networked epidemics, and in this sense is novel.

\subsection{An Auxiliary Control Law for Quarantine Optimization}
\label{subsec:total_quar}
As noted in Section \ref{subsec:empc}, our control scheme assumes the existence of an auxiliary control law which always satisfies the required expectation decay constraint.  While specifying an all-purpose auxiliary control law is outside of the scope of this paper, we demonstrate here how to construct one for the quarantining problem, with the hope that it will provide insight on how to do so in other application areas.  

The policy we design for the quarantine problem, which we refer to as the total quarantine policy, removes nodes which are either exposed to or infected by the disease at each time that the state of the process is observed.  Intuitively, this is a mathematical model for what is implemented in the event of a serious disease outbreak (e.g., the response to the Ebola epidemic of 2014 \cite{Drazen2014}).  This procedure is guaranteed mathematically to eliminate the contagion from the network exponentially quickly provided the control horizon is sufficiently long, as we show in the following result.
\begin{theorem}[Convergence of Total Quarantine Policy]
	\label{thm:total_quar}
	Suppose $\eta_i$ and $\delta_i$ are distinct for all $i \in \Nodes,$ choose $r < \min\{\eta_i, \delta_i\}$ for all $i,$ and $\delt$ to satisfy
	\begin{equation}
		\label{ineq:quar_smpl_bnd}
		\frac{\log(\max\{\delta_i,\eta_i\}) - \log(|\eta_i - \delta_i|)}{\min\{\delta_i,\eta_i\} - r} \leq \delt,
	\end{equation}
	for all $i \in \Nodes.$  Suppose further that at each $t \in \SamplingTimes,$ an action from the total quarantine policy $\utotal,$
	\begin{equation}
	\utotal(X) \triangleq
	\begin{cases}
	\action_i = 1, & \forall i \, \text{s.t.} \, X_i^E + X_i^I > 0,\\
	\action_i = 0, & \forall i \, \text{s.t.} \, X_i^E + X_i^I = 0,
	\end{cases}
	\end{equation}
	is applied and held constant until time $t + \delt.$  Then, the evolution of the $SEIV$ process satisfies
	\begin{equation}
	\begin{aligned}
	&\E_{\utotal}[\InfectExposed(X(t + \delt))|X(t)] \leq  \InfectExposed(X(t)) \me^{-r \delt}
	\end{aligned}
	\end{equation}
	for all $X(t) \in \StateSpace,$ and each $t \in \SamplingTimes.$
\end{theorem}

\begin{proof}
	See Appendix \ref{app:thm:total_quar}. \proofend
\end{proof}

Note that the statement requiring $\eta_i$ and $\delta_i$ to be distinct is one of explanatory convenience.  In particular, our analysis relies on solving a particular system of linear ordinary differential equations, of which $\eta_i$ and $\delta_i$ are the eigenvalues.  By requiring that they be distinct, we simplify the required proofs to only having to consider one possible type of solution.  Of course, the analysis can be done just as easily in the case that $\delta_i = \eta_i.$  However, the principle components of the argument are the same, and thus are left out of the paper.  It should also be noted that the sampling time bound given by \eqref{ineq:quar_smpl_bnd} is conservative.  This inequality was derived for the express purpose of providing a simple inequality which can be checked easily, at the expense of applying approximations.  

The total quarantine policy analyzed in Theorem \ref{thm:total_quar} is conservative.  It removes more nodes from the network than is required to eliminate the epidemic exponentially quickly.  The numerical experiments in Section \ref{subsec:experiments} verify that the proposed EMPC finds actions which are more efficient than the total quarantine policy.

\subsection{A Method for Optimizing Quarantine Use}
\label{subsec:pred_quar}

We now consider the task of of finding good approximate solutions to \eqref{prog:empc_def}.  Note that since the $SEIV$ process evolves on a state space with $4^n$ elements, evaluating the stability constraint $J(X,a) \leq 0$ precisely is in general difficult.  As such, we use the results developed in Section \ref{sec:robust} to evaluate the constraint conservatively.  In particular, at a state $X(0)$ and for a particular control action $a,$ define the function 
\begin{equation}
	\bar{J}(X(0),a) = \psi(X(0),a) - \ell(X(0)) \Exp{-r \delt},
\end{equation}
where the term $\psi(X(0),a)$ is defined as in Theorem \ref{thm:opt_apx}.  Since $\bar{J}(X(0),a) \geq J(X(0),a),$ it follows that $\bar{J}(X(0),a) \leq 0$ implies $J(X(0),a) \leq 0.$  Hence, computing an approximate solution which is feasible to
\begin{equation}
\label{prog:empc_def_apx}
\underset{\action \in \ActionSet}{\text{minimize}} \{ \Cost(\action) \, | \,   \bar{J} (X(t),\action) \leq 0\},
\end{equation}
also computes an approximate solution which is feasible to \eqref{prog:empc_def}.  Note that \eqref{prog:empc_def_apx} is itself still not easy to solve: it is an integer programming problem.  While there are relatively sophisticated methods for solving various classes of integer programs, they all require the presence of some special structure to work well, as integer programs are in general $\NP$-hard.  Moreover, even computing suboptimality bounds for arbitrary integer programming problem is difficult \cite{Burer2012,Lee2012}.  As such, we use a randomized multistart search to approximately solve \eqref{prog:empc_def_apx}.  This is detailed in Algorithm \ref{alg:multistart}.
\begin{algorithm}[!h]
	Initialization:
	\begin{algorithmic}[1]
		\item Define feasible auxiliary solution $\action_{\text{aux}} = \uaux(X(t));$
		\item Initialize set of candidate solutions $\CandidateSolution \triangleq \{\action_{\text{aux}}\};$
		\item Define maximum iteration count $\maxiter;$
		\item Set $k = 0.$
		\item Run main program;
	\end{algorithmic}
	Main Program:
	\begin{algorithmic}[1]
		\While{$k \leq \maxiter$}
		\State Sample $\action$ from distribution with support $\ActionSet;$
		\If{$\action$ is feasible} \label{algstep:qstep1}
			\State $\IdxSet \triangleq \{i \in \Nodes \, | \, \action_i = 1\};$
			\If{$\action - e_i$ is feasible for some $i \in \IdxSet$} \label{algstep:search}
				\State $\action \leftarrow \action - e_i;$
				\State Go to \ref{algstep:qstep1};
			\Else
				\State $\CandidateSolution \leftarrow \CandidateSolution \cup \{\action\}.$
			\EndIf
		\Else
			\State Return $\action$ infeasible;
		\EndIf
		
		\State $k \leftarrow k + 1;$
		\EndWhile
		\State Return $\action^\star = \operatorname{argmin}_{\action \in \CandidateSolution} \{\Cost(\action)\};$
	\end{algorithmic}
	\caption{Multi-Start Local Descent}
	\label{alg:multistart}
\end{algorithm}

While in general it may take a very long time to find the optimal solution, it is theoretically guaranteed that Algorithm \ref{alg:multistart} will find the globally optimal solution of \eqref{prog:empc_def} eventually, and in finite time, so long as $\maxiter$ is set to infinity.  This is moreso a nice theoretical guarantee, than a practically important matter.  Of course, the amount of time required to find a feasible solution of a particular quality will be affected by the particular choice of sampling distribution, and it is likely that the optimal solution will not be found for quite some time.  We detail the particular sampling distribution used in our experiments in Section \ref{subsec:experiments}.

Perhaps most important in practice is that for any considered candidate action, Algorithm \ref{alg:multistart} will terminate at a locally optimal point after at most $O(n^2)$ operations, which can be proven formally by a simple counting argument.  This guarantees that at any given system state, our controller can do better than random guessing quickly, while not necessarily guaranteeing that the optimal action will be found.  Note that by no means is this the only sort of algorithm which can be used here.  Rather, this is the simplest approach which has been found to work well enough to be worth reporting here.  It is expected that in different application domains, researchers may want to investigate the efficacy of different approaches.  Those interested may want to read up on heuristic algorithms for integer programming problems (see, e.g., \cite{Blum2016,Hansen2001}).

\subsection{Numerical Experiments}
\label{subsec:experiments}
We simulate the evolution of the $SEIV$ process under the EMPC defined in Section \ref{subsec:empc}, using the optimization method defined in Section \ref{subsec:pred_quar} with decay rate $r = 0.07$ and $\delt = 0.375,$ where the candidate solutions for Algorithm \ref{alg:multistart} are chosen such that exposed and infected nodes are quarantined independently with probability $0.7,$ and susceptible and vigilant nodes are quarantined independently with probability $0.1.$  All numerical integrations are performed with Matlab's implementation of \texttt{ode45}.  We compare its performance against that of the total quarantine base heuristic given in Section \ref{subsec:total_quar}.  Representative results of our numerical study are given in Figure \ref{fig:sim_results}, which reports the result of a simulation of a $200$-node \Erdos-Reyni random graph with connection probability $0.6,$ and spreading parameters chosen as  $\alpha_i = 0.1,$ $\beta_{ij} = 0.1,$ $\gamma_{ij} = 0.1,$ $\delta_{i} = 1.25,$ $\eta_i = 3.5,$ and $\xi_i = 2,$ where the parameters are equivalued for all edges and nodes so as to be able to fully specify the problem considered.

From the convergence plot given in Figure~\ref{subfig:decay_test}, we see that while the approximations generated by the solutions of \eqref{eq:refined_dyn} are somewhat loose at the beginning of the simulation, they are highly nontrivial by themselves.  Moreover, the optimal upper-bound given by Theorem \ref{thm:opt_apx} improves the approximation further, decreasing the uncertainty of the approximation by a few units in some cases.  At their worst, the optimal bounds constrain the expectation of the process to within an interval of approximately $20$ nodes, corresponding to an uncertainty of approximately $10\%,$ with respect to the number of nodes in the graph.  While this suggests that an ideal controller may be able to attain better performance, it also suggests that possible improvements are limited.  As time passes, the quality of the approximation improves to the point where the upper- and lower- approximations converge, further exemplifying the approximation method's utility.

From the cost plot given in Figure~\ref{subfig:cost}, we can see that the cost of the controller is substantially reduced, with respect to a comparison against the total quarantine base policy.  To some extent, this provides a mathematical validation of the opinions expressed by the medical community, which suggest that quarantining individuals exposed to infectious diseases is \emph{not always} required for effective disease control \cite{Drazen2014}.  Of course, the $SEIV$ model we study here is perhaps too simple to say anything more concrete.

\begin{figure}[!t]
	\begin{subfigure}[b]{\linewidth} 
		\centering\large 
		\includegraphics[width=\textwidth]{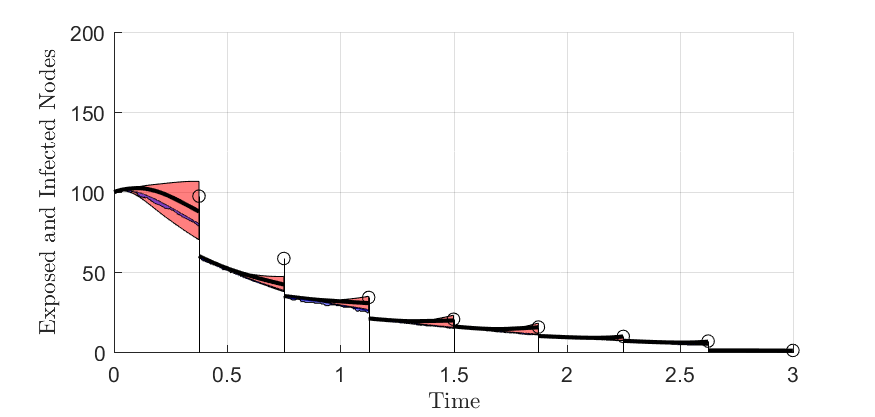}
		\caption{Total number of exposed and infected nodes in the spreading graph under the proposed controller as a function of time.  The shaded red regions give the bounds generated by propagating the dynamics \eqref{eq:refined_dyn}, the open circles indicate the upper-bound on the expected number of exposed and infected nodes guaranteed by the controller, the solid black line gives the optimal upper-bound as computed by Theorem \ref{thm:opt_apx}, the shaded blue regions give the $98\%$ confidence intervals generated from estimating the expectation of the process by producing $10k$ sample paths of the process via Monte Carlo simulation, and then computing $1k$ estimates of the mean via bootstrap sampling from the simulated trajectories.}\label{subfig:decay_test}
	\end{subfigure}
	\begin{subfigure}[b]{\linewidth} 
		\centering\large 
		\includegraphics[width=\textwidth]{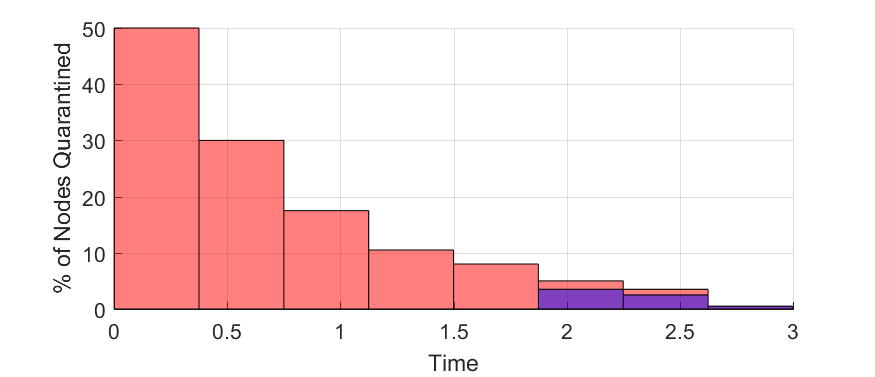}
		\caption{Fraction of nodes held in quarantine as a function of time.  The red regions indicate the total fraction of nodes quarantined under the total quarantine policy, the blue regions indicate the total fraction of nodes quarantined under the proposed controller.  In this particular simulation, a quarantine is not required at all until the number of infections in the graph is very small.} \label{subfig:cost}
		\label{fig:cost_test} 
	\end{subfigure}
	\caption{Plots numerically evaluating the performance of the proposed controller.  Figure \ref{subfig:decay_test} shows that the proposed controller induces exponential elimination of the epidemic.  Figure \ref{subfig:cost} demonstrates that the stochastic optimization method proposed in Algorithm \ref{alg:multistart} can significantly reduce the number of nodes required to be quarantine in order to guarantee the desired elimination rate.} \label{fig:sim_results}
\end{figure}

\section{Summary and Future Work} 
\label{sec:summary}

In this paper, we have developed a robust economic model predictive controller for the mitigation of diseases modeled by $SEIV$ processes.  In addressing this problem, we provided a novel robust moment closure technique, which guarantees that the engendered approximations always give rigorous, nontrivial upper- and lower- bounds of the processes' compartmental membership probabilities.  We have also shown how to analyze the convergence properties of the proposed EMPC in order to guarantee that the elimination time of the process is appropriately small.  In so doing, we have provided a useful step in the process to developing control techniques for continuous-time stochastic networked epidemics which provide statistical guarantees about the evolution of the process.  Still, there is much work to be done.

Perhaps an obvious criticism of the basic $SEIV$ model is that - due to modeling the underlying contact and compartment transition processes as Poisson processes - the holding time distributions for each node's compartmental memberships are constricted to be exponential, which may not be reflective of what is seen in actual diseases.  As such, it would be of particular interest to extend the analysis of the process presented in this paper to a non-Markovian setting, such that non-exponential holding time distributions can be accommodated.  It is also of interest to determine which types of optimization methods work best for controllers designed for applications in different domains, where the objectives and epidemic models may be different.  In particular, it seems of interest to extend the framework presented here to the types of problems which have been studied exclusively in the mean-field regime.  This paper provides a step in the direction of solving such problems; we believe that continued interest from the community will lead to interesting solutions.

\appendix
\subsection{Derivation of $SEIV$ Expectation Dynamics}
\label{app:dynamics}

Since the integrands of the \Ito integrals of \eqref{eq:SEIV_process} are always finite, the process is square integrable.  As such, a consequence of \Itos lemma is that the expectation operator and the \Ito integral commute (see, e.g., \cite[Theorem 3.20]{Hanson2007}), and the expectation of the probability measures of the Poisson processes become the rates of the process.  Carrying this computation through, we may take the expectation of both sides of \eqref{eq:SEIV_process} to arrive at the integral equations
\begin{equation}
\label{eq:SEIV_integrals}
\begin{aligned}
&\E[X_i^{S}(t)] =\int_{0}^{t} \E[X_i^V(h)] \alpha_i - \E[X_i^{S}(h)] \xi_i\\
&\hspace{40 pt}- \sum_{j \in \Neighbors_i}( \E[X_i^S X_j^E(h)] \beta_{ij} + \E[X_i^S X_j^I(h)] \gamma_{ij}) \deriv h,\\
&\E[X_i^{E}(t)] = \int_{0}^{t} \sum_{j \in \Neighbors_i} (\E[X_i^S X_j^E(h)] \beta_{ij} + \E[X_i^SX_j^I(h)] \gamma_{ij}) \\
&\hspace{40 pt} - \E[X_i^E(h)] \delta_{i} \deriv h,\\
&\E[X_i^{I}(t)] = \int_{0}^{t} \E[X_i^E(h)] \delta_{i} - \E[X_i^I(h)] \eta_i \deriv h,\\
&\E[X_i^{V}(t)] =  \int_{0}^{t} \E[X_i^I(h)] \eta_{i} + \E[X_i^{S}(h)] \xi_i \\
&\hspace{40 pt}- \E[X_i^{V}(h)] \alpha_i \deriv h.
\end{aligned}
\end{equation}
Note that only ordinary Riemann integrals remain in \eqref{eq:SEIV_integrals}, so we may apply the fundamental theorem of calculus (see, e.g., \cite[Theorem 6.20]{Rudin1976}) to arrive at the system of ordinary, nonlinear differential equations \eqref{eq:SEIV_diff}. \proofend

\subsection{Two-Sided Multivariate Comparison Lemma}
\label{app:lem:comparison}
\begin{lem}[Two-Sided Multivariate Comparison Lemma] 
	\label{lem:comparison}
	Consider a system of differential equations
	\begin{equation}
	\dot{x} = f(x)
	\end{equation}
	with $x \in \real^p,$ $f : \real^p \rightarrow \real^p,$ and possessing a unique, continuously differentiable solution $x(t).$  Suppose $\tilde{f}$ and $\utilde{f}$ are Lipschitz continuous vector functions defined on $\real^{p \times 2},$ where for each component $i,$
	\begin{equation}
	\label{ineq:dyn_upper_bound}
	f_i(z) \leq \tilde{f}_i(\utilde{z},\tilde{z})\\
	\end{equation}
	holds everywhere on the set
	\begin{equation*}
	\tilde{\mathcal{Z}}_i \triangleq \{(\utilde{z}, z, \tilde{z}) \in \real^{(p \times 3)} \, | \,  \utilde{z}_i \leq z_i = \tilde{z}_i, \utilde{z}_j \leq z_j \leq \tilde{z}_j, \forall j \neq i\},
	\end{equation*}
	and the inequality  
	\begin{equation}
	\label{ineq:dyn_lower_bound}
	\utilde{f}_i(\utilde{z},\tilde{z}) \leq f_i(z)
	\end{equation}
	holds everywhere on the set
	\begin{equation*}
	\utilde{\mathcal{Z}}_i \triangleq \{(\utilde{z}, z, \tilde{z}) \in \real^{(p \times 3)} \, | \,  \utilde{z}_i = z_i \leq \tilde{z}_i, \utilde{z}_j \leq z_j \leq \tilde{z}_j, \forall j \neq i\}.
	\end{equation*}
	Then, the solutions to the system
	\begin{equation}
	\label{eq:f_robust_dyn}
	\begin{aligned}
	\dot{\tilde{x}}= \tilde{f}(\utilde{x},\tilde{x}),\\
	\dot{\utilde{x}} = \utilde{f}(\utilde{x},\tilde{x}),
	\end{aligned}
	\end{equation}
	with initial conditions $\utilde{x}(0) = x(0) = \tilde{x}(0)$, satisfy
	\begin{equation}
	x_i(t) \in [\utilde{x}_i(t), \tilde{x}_i(t)]
	\end{equation}
	for all $i$ and $t \geq 0.$ 
\end{lem}

\begin{proof}
	The proof proceeds by a sequence of contradiction arguments, which in particular use the fact that the solutions $\utilde{x}(t),$ $x(t),$ and $\tilde{x}(t)$ are continuously differentiable along with the inequalities \eqref{ineq:dyn_lower_bound} and \eqref{ineq:dyn_upper_bound} in order to demonstrate that the inclusions $x_i(t) \in [\utilde{x}_i(t), \tilde{x}_i(t)]$ hold for all $i \in [p],$ and all times $t \geq 0.$  Note that the continuous differentiability of the solutions follows immediately from Lipschitz continuity of the dynamics (see, e.g., \cite{Khalil2002}).
	
	Suppose for purposes of contradiction that $\utilde{x}(t) \curlyleq x(t) \curlyleq \tilde{x}(t)$ does not hold for all time.  That then implies that there is some $t$ at which either $x_i(t) > \tilde{x}_i(t)$ or $x_i(t) < \utilde{x}_i(t)$ occurs.  Let $\contratime$ be the first time at which such an event occurs.  Suppose for now that $x_i(\contratime) > \tilde{x}_i(\contratime)$ occurs; we argue the other case analogously.  By the continuity and differentiability of $x_i$ and $\tilde{x}_i,$ it then also holds that there exists some time $\precontratime$ such that $\precontratime < \contratime,$ $x_i(\precontratime) = \tilde{x}_i(\precontratime),$ and $\dot{x}_i(\precontratime) > \dot{\tilde{x}}_i(\precontratime).$  However, since it holds that $\utilde{x}_i(\precontratime) \leq x_i(\precontratime) = \tilde{x}_i(\precontratime),$ and we have that $\utilde{x}_j(\precontratime) \leq x_j(\precontratime) \leq \tilde{x}_j(\precontratime),$ we have by assumption that $\dot{x}_i(\precontratime) \leq \dot{\tilde{x}}_i(\precontratime).$  This is a contradiction.  The case in which $x_i(\contratime) < \utilde{x}_i(\contratime)$ occurs can be handled by similar arguments.  This completes the proof.
\proofend
\end{proof}

\subsection{Proof that the dynamics \eqref{eq:SEIV_frechet_apx_dyn} satisfy Lemma \ref{lem:comparison}}
\label{app:veracity}
	Note that by using a construction such as \cite[Section 5]{Sahneh2013}, the $SEIV$ process can be represented as a $4^n$-dimensional time-homogeneous Markov process for fixed set of spreading parameters.  As such, the compartmental membership probabilities generated from any particular initial state $X$ can be represented as sums of states evolving as solutions to a $4^n$-dimensional linear system, and so are unique and continuously differentiable.  Likewise, since each term of the dynamics \eqref{eq:SEIV_frechet_apx_dyn} is a sum of Lipschitz continuous functions, it follows that the dynamics \eqref{eq:SEIV_frechet_apx_dyn} are also Lipschitz continuous.  It remains to verify that the inequalities required by Lemma \ref{lem:comparison} are satisfied on the appropriate subsets of the state space.
	
	Choose some node label $i \in \Nodes$ and some compartmental label $C \in \mathcal{L}.$  We wish to show that
	\begin{equation} \label{ineq:proof1}
		\max_{\xi \in \Xi(x)} \left\{\ddtop{\E_{\xi}[X_i^C]} \right\} \leq \dot{\tilde{x}}_i^C(\utilde{x},\tilde{x})
	\end{equation}
	holds everywhere on
	\begin{equation*}
		\begin{aligned}
		&\tilde{\mathcal{X}}_i^C \triangleq \{(\utilde{x},x,\tilde{x}) \in \real^{(\Nodes \times \mathcal{L}) \times 3} \\
		& \hspace{10 pt} | \, \utilde{x}_i^C \leq x_i^{C} = \tilde{x}_i^C, \utilde{x}_j^{C^\prime} \leq x_j^{C^\prime} \leq \tilde{x}_j^{C^\prime}, (j,C^\prime) \in \Nodes \times \mathcal{L} \},
		\end{aligned}
	\end{equation*}
	where the set $\Xi(x)$ is the set of all probability measures with first moments $x_i^{C} = \E_{\xi}[X_i^{C}],$ and where the term $\dot{\tilde{x}}_i^C(\utilde{x},\tilde{x})$ is a shorthand reference to the dynamics \eqref{eq:SEIV_frechet_apx_dyn}. Consider all terms of the function $\ddtop{\E_{\xi}[X_i^C]}$ with positive coefficients; they can be written as $\kappa \E_{\xi}[X_i^{C} X_j^{C^\prime}]$ for some $\kappa \geq 0.$  Each is bounded above by $\FrechetUpper(\tilde{x}^{C}_i,\tilde{x}^{C^\prime}_{j})$ by the \Frechet inequality \eqref{ineq:frechet} for any measure with expectation $x.$  Likewise, consider all terms of the function $\ddtop{\E_{\xi}[X_i^C]}$ with negative coefficients; they can be written as $\kappa \E_{\xi}[X_i^{C} X_j^{C^\prime}]$ for some $\kappa \leq 0.$  From the \Frechet inequality \eqref{ineq:frechet}, we have $\kappa \FrechetLower(\tilde{x}_i^{C},\utilde{x}_j^{C^\prime}) \geq \kappa \E_{\xi}[X_i^{C} X_j^{C^\prime}],$ so long as $\tilde{x}_i^{C} = \E_{\xi} [X_i^C],$ which is precisely the case for points on $\tilde{\mathcal{X}}_i^C.$  Hence, the inequality \eqref{ineq:proof1} holds on $\tilde{\mathcal{X}}_i^C,$ as claimed.  We can show that the inequality
	\begin{equation*}
		 \dot{\utilde{x}}_i^C(\utilde{x},\tilde{x}) \leq \min_{\xi \in \Xi(x)} \left\{\ddtop{\E_{\xi}[X_i^C]} \right\}
	\end{equation*}
	holds everywhere on the set
	\begin{equation*}
	\begin{aligned}
	&\utilde{\mathcal{X}}_i^C \triangleq \{(\utilde{x},x,\tilde{x}) \in \real^{(\Nodes \times \mathcal{L}) \times 3} \\
	& \hspace{10 pt} | \, \utilde{x}_i^C = x_i^{C} \leq \tilde{x}_i^C, \utilde{x}_j^{C^\prime} \leq x_j^{C^\prime} \leq \tilde{x}_j^{C^\prime}, (j,C^\prime) \in \Nodes \times \mathcal{L} \}
	\end{aligned}
	\end{equation*}
	by similar arguments, which completes the proof.
	\proofend

\subsection{Proof of Proposition \ref{prop:tighter}}
\label{app:prop:tighter}

This result relies directly on a well-known comparison lemma from the theory of monotone dynamical systems, the Kamke-\Muller lemma (see, e.g. \cite{Smith2004}).  In its most basic form, the Kamke-\Muller lemma gives comparisons between the solutions of one dynamical system, evolving from two distinct initial conditions which satisfy some ordering.  However, it is well known that this type of comparison can be made for the solutions of two distinct dynamical systems, provided certain ordering conditions hold.  Specifically, the result we use in this paper is stated as follows, where we have adapted material from~\cite[Section 3.1]{Smith2004} to the notation used here:
\begin{lem}[Extended Kamke-\Muller Lemma]
	\label{lem:kamke-muller}
	Let $f$ and $g$ be Lipschitz continuous vector functions on $\real^p,$ and consider the nonlinear dynamical system
	\begin{equation*}
	\begin{aligned}
	\dot{y} &= f(y)\\
	\dot{z} &= g(z).
	\end{aligned}
	\end{equation*}
	If for each $i,$ the inequality $f_i(\ubar{q}) \leq g_i(\bar{q})$ holds on
	$$\mathcal{Q}_i \triangleq \left\{(\ubar{q}, \bar{q}) \in \real^{p \times 2} \, | \, \ubar{q}_i = \bar{q}_i, \ubar{q}_j \leq \bar{q}_j, j \neq i \right\},$$
	and it holds that $y(t_0) = z(t_0),$ then it also holds that $y_i(t) \leq z_i(t)$ for all $t \geq t_0,$ and all $i \in [p].$
\end{lem}

Now, consider taking $y^T = [\bar{x}^T,-\ubar{x}^T]^T,$ and $z^T = [\tilde{x}^T,-\utilde{x}^T]^T,$ and let $p = 4n,$ where we let each $j \in [4n]$ represent exactly one $(i,C) \in \Nodes \times \mathcal{L}.$  If for each $(i,C) \in \Nodes \times \mathcal{L},$ we have that $[\ubar{f}_i^C(\ubar{q}),\bar{f}_i^C(\ubar{q})] \subseteq [\utilde{f}_i^C(\bar{q}), \tilde{f}_i^C(\bar{q})],$ on the corresponding set $\mathcal{Q}_j,$ and $(\ubar{x}(0),\bar{x}(0)) = (\utilde{x}(0),\tilde{x}(0)),$ Lemma \ref{lem:kamke-muller} implies $[\ubar{x}_i^C(t), \bar{x}_i^C(t)] \subseteq [\utilde{x}_i^C(t), \tilde{x}_i^C(t)]$ for all $(i,C) \in \Nodes \times \mathcal{L},$ and for all $t \geq 0,$ as claimed. \proofend

\subsection{Proof of Theorem \ref{thm:refined}}
\label{app:thm:refined}
It is easy to check that for all nodes $i,$ and all compartmental labels $C \in \mathcal{L} \triangleq \{S,E,I,V\},$ we have that
\begin{equation*}
	[\ubar{f}_i^C(\ubar{x},\bar{x}),\bar{f}_i^C(\ubar{x},\bar{x})] \subseteq [\utilde{f}_i^C(\utilde{x},\tilde{x}), \tilde{f}_i^C(\utilde{x},\tilde{x})],
\end{equation*}
holds everywhere on $\bar{\mathcal{X}}_i^C \cup \ubar{\mathcal{X}}_i^C$ as defined in Proposition \ref{prop:tighter}, which then gives that $[\ubar{x}_i^C(t),\bar{x}_i^C(t)] \subseteq [\utilde{x}_i^C(t),\tilde{x}_i^C(t)]$ holds for any initial condition $\utilde{x}(0) = \tilde{x}(0) = \ubar{x}(0) = \bar{x}(0) = x(0),$ all $t \geq 0,$ and all $(i,C) \in \Nodes \times \mathcal{L}.$  Since including the complement bounding operators still gives valid over- and under-approximations of the terms of $\ddtop{\E[X_i^C]}$ on the relevant subsets of the state space, this certifies that the system \eqref{eq:refined_dyn} generates solutions which satisfy $\E[X_i^C (t)] \in [\ubar{x}_i^C(t), \bar{x}_i^C(t)]$ for all $t \geq 0,$ and all $(i,C) \in \Nodes \times \mathcal{L}.$

To demonstrate that $[\ubar{x}(t), \bar{x}(t)] \subseteq [0,1]^{\Nodes \times \mathcal{L}}$ holds for all $t,$ we make an appeal to the continuity of the solutions $(\ubar{x}(t),\bar{x}(t))$ as well the values of $\dot{\ubar{x}}$ and $\dot{\bar{x}}$ on the sets of states $(\ubar{x}, \bar{x}),$ such that $\ubar{x}_i^{C} = 0$ or $\bar{x}_i^{C} = 1$ for some $i$ or $C.$  Suppose for purposes of contradiction that $\ubar{x}_i^C(\contratime) < 0$ for some time $\contratime.$  By continuity of $\ubar{x}_i^C,$ it must then be the case that there exists some time $\precontratime < \contratime$ such that $\ubar{x}_i^{C}(\precontratime) = 0$ and $\dot{\ubar{x}}_i^{C}(\precontratime) < 0.$  However, evaluating the expression for $\dot{\ubar{x}}_i^{C}$ with $\ubar{x}_i^{C} = 0,$ has $\dot{\ubar{x}}_i^{C} \geq 0.$  This is a contradiction, and as such, proves that $\ubar{x}_i^C(t) \geq 0$ for all $t \geq 0.$ 

Suppose now that $\bar{x}_i^C(\contratime) > 1$ for some finite time $\contratime.$  By continuity of $\bar{x}_i^C(\contratime),$ it must then be the case that there exists some time $\precontratime < \contratime$ such that $\bar{x}_i^{C}(\precontratime) = 1$ and $\dot{\bar{x}}_i^{C}(\precontratime) > 0.$  However, evaluating the expression for $\dot{\bar{x}}_i^{C}$ with $\bar{x}_i^{C} = 1,$ has $\dot{\bar{x}}_i^{C} \leq 0.$  This is a contradiction, and as such, proves that $\bar{x}_i^C(t) \leq 1$ for all $t \geq 0.$  Since each component of the dynamics \eqref{eq:refined_dyn} is a sum of Lipschitz continuous functions, they are Lipschitz continuous. This completes the proof. \proofend

\subsection{Proof of Theorem \ref{thm:opt_apx}}
\label{app:thm:opt_apx}

Let  $\Dists_{\Graph}$ be the set of all possible marginal compartmental membership probabilities for the graph, i.e.
\begin{equation}
\Dists_{\Graph} \triangleq \left\{y \in \realnonnegative^{\Nodes \times \mathcal{L}} \, | \, \sum_{C \in \mathcal{L}} y_i^C = 1, \forall \, i \in \Nodes \right \}
\end{equation}
Define $\Lambda_{t}$ as the  set of marginal compartmental membership probabilities permitted by the approximations generated by integrating the dynamics \eqref{eq:refined_dyn} with initial condition $\ubar{x}(0) = X(0) = \bar{x}(0)$ over the interval $[0, t],$ i.e.
\begin{equation*}
\Lambda_{t} \triangleq \left \{y \in \Dists_{\Graph} | y_i^{C} \in [\ubar{x}_i^C(t),\bar{x}_i^C(t)], \forall i \in \Nodes, C \in \mathcal{L}  \right \}.
\end{equation*}

Since $\ell(X)$ is a sum of indicator random variables, it follows that $\E[\ell(X(t)) | X(0)]$ can take the value $\ell(y)$ for any $y \in \Lambda_t.$  By maximizing over all such $y,$ it follows that $\E[\ell(X(t)) | X(0)] \leq \max_{y \in \Lambda_{t}} \{\ell(y)\}$ holds, and is tight because $y \in \Lambda_t.$  It remains to show that $\max_{y \in \Lambda_{t}} \{\ell(y)\}$ evaluates to the right hand side of \eqref{ineq:opt_bound}. 

To do this, we note that $\max_{y \in \Lambda_{t}} \{\ell(y)\}$ can be decomposed as the sum of the $n$ linear programs
\begin{subequations}
	\label{prog:max_ell}
	\begin{align}
	& \underset{y_i \in \realnonnegative^{\mathcal{L}} }{\text{maximize}}
	& & y_i^E + y_i^I  \\
	& \text{subject to}
	& & y_i^C \in [\ubar{x}_i^C,\bar{x}_i^C], \forall i \in \Nodes, C \in \mathcal{L} \label{con:interval}\\
	&&& \sum_{C \in \mathcal{L}} y_i^C = 1 \label{eq:sum_con}
	\end{align}
\end{subequations}
We now solve \eqref{prog:max_ell} analytically.  By considering the constraint \eqref{eq:sum_con}, we have that at all feasible points of \eqref{prog:max_ell} satisfy $y_i^E + y_i^I = 1 - y_i^{S} - y_i^{V}.$  As such, if $\bar{x}_i^E + \bar{x}_i^{I}$ and $1 - \ubar{x}_i^{V} + \ubar{x}_i^{S}$ take distinct values, only the smaller of the two values is attainable on the feasible polytope.  As the objective is monotonically increasing in $y_i^E$ and $y_i^I,$ it then follows that the optimal value of \eqref{prog:max_ell} is bounded above by $\min\{\bar{x}_i^E + \bar{x}_i^I, 1- \ubar{x}_i^{V} - \ubar{x}_i^{S} \}.$  Noting that the program \eqref{prog:max_ell} is guaranteed to be feasible since the interval constraints \eqref{con:interval} are generated so as to contain the true underlying marginal probabilities, we can finish the argument by showing that there exists a feasible point which attains the value 
$\min\{\bar{x}_i^E + \bar{x}_i^I, 1- \ubar{x}_i^{V} - \ubar{x}_i^{S} \}.$

Suppose that $\bar{x}_i^E + \bar{x}_i^I > 1 - \ubar{x}_i^V - \ubar{x}_i^S.$  It follows that $\bar{x}_i^E + \bar{x}_i^I + \ubar{x}_i^V + \ubar{x}_i^S > 1,$ from which feasibility of \eqref{prog:max_ell} and the fact that $y_i^S$ and $y_i^V$ are lower bounded by $\ubar{x}_i^S$ and $\ubar{x}_i^{V}$ (respectively) implies that there exists a $y_i$ such that $y_i^E + y_i^I = 1 - \ubar{x}_i^V + \ubar{x}_i^S.$  This point attains the value $\min\{\bar{x}_i^E + \bar{x}_i^I, 1- \ubar{x}_i^{V} - \ubar{x}_i^{S} \}.$  The case in which $\bar{x}_i^E + \bar{x}_i^I \leq 1 - \ubar{x}_i^V - \ubar{x}_i^S$ can be handled by similar arguments.  This ends the proof. \proofend

\subsection{Proof of Theorem \ref{thm:online}}
\label{app:thm:online}
Our argument proceeds by demonstrating that the sequence of controls realized by the proposed EMPC method induce the desired decay property.  Principally, the proof relies on an induction, and an application of the tower property of conditional expectations \cite[Proposition 13.2.7]{Rosenthal2013}.  

We have by construction that the inequality
\begin{equation}
	\E_{\uaux}[\InfectExposed(X(t + \delt)) | X(t)] \leq \InfectExposed(X(t)) \me^{-r \delt}
\end{equation}
holds.  Since the optimization routine will only pass back solutions which satisfy the stability constraint, as it will at worst pass back the action taken by the auxiliary policy, we have that
\begin{equation}
\E_{\umpc}[\InfectExposed(X(t + \delt)) | X(t)] \leq \InfectExposed(X(t)) \me^{-r \delt}
\end{equation}
holds as well.  We use this in an induction to prove the inequality demanded by the theorem's statement.  Take
\begin{equation}
	\E_{\umpc}[\InfectExposed(X(\delt))|X(0)] \leq \InfectExposed(X(0)) \me^{-r \delt}
\end{equation}
as a base for induction, and for an induction hypothesis that
\begin{equation}
	\label{ineq:induct_hyp}
	\E_{\umpc} [\InfectExposed(X(k \delt))|X(0)] \leq \InfectExposed(X(0)) \me^{-r k \delt}
\end{equation}
holds for some arbitrary positive integer $k.$  We now show that this implies that
\begin{equation*}
	\E_{\umpc}[\InfectExposed(X( (k+1) \delt)) | X(0)] \leq \InfectExposed(X(0)) \me^{-r (k+1) \delt}.
\end{equation*}

By expanding the conditioning in accordance with the tower property of conditional expectation \cite[Proposition 13.2.7]{Rosenthal2013}, we have that the identity
\begin{equation}
	\label{eq:induct_ident}
	\begin{aligned}
	&\E_{\umpc}[\InfectExposed(X((k+1)\delt)) | X(0)]\\
	&\hspace{20pt} =\E_{\umpc}[\E[\InfectExposed(X((k+1) \delt))| X(k \delt), X(0)]
	\end{aligned}
\end{equation}
holds.  From definition, we have the inequality
\begin{equation*}
	\E_{\umpc}[\InfectExposed(X((k+1) \delt))| X(k \delt), X(0)] \leq \InfectExposed(X(k \delt)) \me^{-r \delt},
\end{equation*}
which when applied to the identify \eqref{eq:induct_ident} gives
\begin{equation*}
	\begin{aligned}
	&\E_{\umpc}[\E[\InfectExposed(X((k+1) \delt))| X(k \delt), X(0)] ] \\
	&\hspace{85 pt} \leq \E_{\umpc}[ \InfectExposed(X(k \delt))| X(0)] \me^{-r \delt}.
	\end{aligned}
\end{equation*}
Now, the induction hypothesis \eqref{ineq:induct_hyp} yields the inequality
\begin{equation*}
	\E_{\umpc}[\InfectExposed(X((k+1) \delt)) | X(0)] \me^{-r \delt} \leq \InfectExposed(X(0)) \me^{-r (k+1) \delt},
\end{equation*}
which shows that \eqref{ineq:alg1_decay_samples} holds for any $t \in \SamplingTimes.$  Noting that this final inequality is tight in the case where \eqref{ineq:e_decay} is tight at every sampling time completes the proof. \proofend

\subsection{Proof of Theorem \ref{thm:total_quar}}
\label{app:thm:total_quar}
The essence of this proof is in demonstrating that the total quarantine policy induces sufficient negative drift in the process so as to guarantee the expectation decay stated in the theorem's hypothesis.  To accomplish this, we analyze the evolution of the upper-bounds on the compartmental membership probabilities for node $i$ and compartments $E$ and $I.$  Writing their dynamics down from \eqref{eq:SEIV_frechet_apx_dyn} with $\beta_{ij}$ and $\gamma_{ij}$ set to $0$ for all $j,$ we have that the linear system
\begin{equation}
	\label{eq:EI_apx}
	\begin{aligned}
	\dot{\bar{x}}_i^{E} &= - \delta_i \bar{x}_i^E,\\
	\dot{\bar{x}}_i^{I} &= \delta_i \bar{x}_i^E - \eta_{i} \bar{x}_i^I, \\
	\end{aligned}
\end{equation}
describes the behavior of the approximating system under the total quarantine policy.  Using standard solution techniques from the theory of linear ordinary differential equations to solve \eqref{eq:EI_apx} with initial conditions $\bar{x}^E(t)$ and $\bar{x}^I(t)$ and $\eta_i,$ $\delta_i$ distinct gives the solutions:
\begin{equation*}
	\begin{aligned}
	&\bar{x}_i^I(\delt|t) + \bar{x}_i^E (\delt|t) =  \\
	&\bar{x}_i^I(t) \me^{- \eta_i \delt} + \frac{\eta_i}{\eta_i - \delta_i} \bar{x}_i^{E}(t) \me^{- \delta_i \delt} - \frac{\delta_i}{\eta_i - \delta_i} \bar{x}_i^{E}(t) \me^{- \eta_i \delt}.
	\end{aligned}
\end{equation*}
Because we observe the state $X(t)$ at each sampling time $t \in \SamplingTimes,$ we have only two possible initial conditions for each $i:$ $\bar{x}_i^E(t) = 1$ and $\bar{x}_i^I(t) = 0$ or $\bar{x}_i^E(t) = 0$ and $\bar{x}_i^I(t) = 1.$

In the first case, we need to verify that there exists some $r > 0$ and $\delt$ which satisfy the decay constraint
\begin{equation}
	\frac{\eta_i}{\eta_i - \delta_i} \me^{- \delta_i \delt}  - \frac{\delta_i}{\eta_i - \delta_i} \me^{- \eta_i \delt} \leq \me^{-r \delt}.
\end{equation}
By assuming $\eta_i > \delta_i > r$ rearranging terms, approximating the resulting inequality, and taking logarithms, we get
\begin{equation}
	\frac{\ln(\eta_i) - \ln(\eta_i - \delta_i)}{\delta_i - r} \leq \delt.
\end{equation}
Similarly, by assuming $\delta_i > \eta_i > r,$ we get the inequality
\begin{equation}
	\frac{\ln(\delta_i) - \ln(\delta_i - \eta_i)}{\eta_i - r} \leq \delt.
\end{equation}
Considering both inequalities together verifies that the inequality stated by the theorem's hypothesis, i.e.
\begin{equation}
	\label{ineq:rate_time}
	\frac{\ln(\max\{\eta_i,\delta_i\}) - \ln(|\eta_i - \delta_i|)}{\min\{\eta_i,\delta_i\} - r} \leq \delt
\end{equation}
suffices to demonstrate that when \eqref{eq:EI_apx} is initialized with  $\bar{x}_i^E(t) = 0$ and $\bar{x}_i^I(t) = 1,$ we satisfy the desired exponential decay inequality.  In the case where  $\bar{x}_i^E(t) = 1$ and $\bar{x}_i^I(t) = 0,$ we have that the exponential decay inequality is satisfied for any $t,$ and all $r < \eta_i,$ which is implied by the inequality claimed by the hypothesis.

Recollecting our argument, we see that for $r$ and $\delt$ satisfying \eqref{ineq:rate_time}, we have the decay inequality 
\begin{equation*}
	\bar{x}_i^E(\delt|t) + \bar{x}_i^{I} (\delt|t) \leq (\bar{x}_i^{E}(t) + \bar{x}_i^{I}(t)) \me^{-r \delt}.
\end{equation*}
By summing over all of the nodes in the network, we get
\begin{equation*}
	\begin{aligned}
	\sum_{i = 1}^n \bar{x}_i^{E}(\delt|t) + \bar{x}_i^{I}(\delt|t) \leq \left(\sum_{i = 1}^{n} \bar{x}_i^{I}(t) + \bar{x}_i^{E}(t) \right) \me^{- r \delt}
	\end{aligned}
\end{equation*}
which since $\InfectExposed(X(t)) =  \left(\sum_{i = 1}^{n} \bar{x}_i^{I}(t) + \bar{x}_i^{E}(t) \right)$ and $\E[\InfectExposed(X(t+\delt)) | X(t)] \leq \left(\sum_{i = 1}^{n} \bar{x}_i^{I}(\delt|t) + \bar{x}_i^{E}(\delt|t) \right)$ hold, together imply
\begin{equation*}
	\E_{\utotal}[\InfectExposed(X(t+\delt)) | X(t)] \leq \InfectExposed(X(t))\me^{- r \delt}
\end{equation*}
which was sought.  This concludes the proof. \proofend

\bibliography{library}

\begin{thebibliography}{10}

\bibitem{Bailey1975}
N.~T.~J. Bailey, {\em {The Mathematical Theory of Infectious Diseases and Its
  Applications}}.
\newblock High Wycombe, U.K.: Charles Griffin {\&} Company, Ltd., 2nd~ed.,
  1975.

\bibitem{Preciado2014}
V.~M. Preciado, M.~Zargham, C.~Enyioha, A.~Jadbabaie, and G.~J. Pappas,
  ``{Optimal Resource Allocation for Network Protection Against Spreading
  Processes},'' {\em IEEE Transactions on Control of Network Systems}, vol.~1,
  no.~1, pp.~99--108, 2014.

\bibitem{Nowzari2017}
C.~Nowzari, V.~M. Preciado, and G.~J. Pappas, ``{Optimal Resource Allocation
  for Control of Networked Epidemic Models},'' {\em IEEE Transactions on
  Control of Network Systems}, vol.~4, no.~2, pp.~159--169, 2017.

\bibitem{Ogura2016}
M.~Ogura and V.~M. Preciado, ``{Stability of Spreading Processes over
  Time-Varying Large-Scale Networks},'' {\em IEEE Transactions on Network
  Science and Engineering}, vol.~3, no.~1, pp.~44--57, 2016.

\bibitem{Wang2014}
S.~Wang, M.~H.~R. Khouzani, B.~Krishnamachari, and F.~Bai, ``{Optimal Control
  for Epidemic Routing of Two Files with Different Priorities in Delay Tolerant
  Networks},'' in {\em IEEE American Control Conference}, (Chicago, IL, USA),
  pp.~1387--1392, 2015.

\bibitem{Lee2016a}
P.~Lee, A.~Clark, B.~Alomair, L.~Bushnell, and R.~Poovendran, ``{Adaptive
  Mitigation of Multi-Virus Propagation: A Passivity-Based Approach},''
  vol.~5870, no.~c, pp.~1--11, 2016.

\bibitem{Watkins2015}
N.~J. Watkins, C.~Nowzari, V.~M. Preciado, and G.~J. Pappas, ``{Optimal
  resource allocation for competing epidemics over arbitrary networks},'' in
  {\em Proceedings of the American Control Conference}, pp.~1381--1386, 2015.

\bibitem{Eshghi2016}
S.~Eshghi, M.~H. Khouzani, S.~Sarkar, and S.~S. Venkatesh, ``{Optimal Patching
  in Clustered Malware Epidemics},'' {\em IEEE/ACM Transactions on Networking},
  vol.~24, no.~1, pp.~283--298, 2016.

\bibitem{Eshghi2017}
S.~Eshghi, S.~Sarkar, and S.~S. Venkatesh, ``{Visibility-Aware Optimal
  Contagion of Malware Epidemics},'' {\em IEEE Transactions on Automatic
  Control}, vol.~62, no.~October, pp.~5205--5212, 2017.

\bibitem{Hota2017}
A.~R. Hota and S.~Sundaram, ``{Optimal network topologies for mitigating
  security and epidemic risks},'' {\em 54th Annual Allerton Conference on
  Communication, Control, and Computing, Allerton 2016}, pp.~1129--1136, 2017.

\bibitem{Pare2017}
P.~E. Par{\'{e}}, J.~Liu, C.~L. Beck, A.~Nedi{\'{c}}, and T.~Başar,
  ``{Multi-Competitive Viruses over Static and Time – Varying Networks},''
  pp.~1--16.

\bibitem{Drakopoulos2014}
K.~Drakopoulos, A.~Ozdaglar, and J.~Tsitsiklis, ``{An Efficient Curing Policy
  for Epidemics on Graphs},'' {\em IEEE Transactions on Network Science and
  Engineering}, vol.~1, no.~2, pp.~67--75, 2015.

\bibitem{Scaman2016}
K.~Scaman, A.~Kalogeratos, and N.~Vayatis, ``{Suppressing Epidemics in Networks
  Using Priority Planning},'' {\em IEEE Transactions on Network Science and
  Engineering}, vol.~3, no.~4, pp.~271--285, 2016.

\bibitem{Watkins2017}
N.~J. Watkins, C.~Nowzari, and G.~J. Pappas, ``{Inference, Prediction, and
  Control of Networked Epidemics},'' in {\em IEEE American Control Conference},
  (Seattle, WA, USA), pp.~5611 -- 5616, 2017.

\bibitem{Wei2013}
X.~Wei, N.~C. Valler, B.~{Aditya Prakash}, I.~Neamtiu, M.~Faloutsos, and
  C.~Faloutsos, ``{Competing Memes Propagation on Networks: A Network Science
  Perspective},'' {\em IEEE Journal on Selected Areas in Communications},
  vol.~31, no.~6, pp.~1049--1060, 2013.

\bibitem{Keeling2005}
M.~J. Keeling and K.~T. Eames, ``{Networks and epidemic models},'' {\em Journal
  of The Royal Society Interface}, vol.~2, no.~4, pp.~295--307, 2005.

\bibitem{Eksin2017}
C.~Eksin, J.~S. Shamma, and J.~S. Weitz, ``{Disease dynamics in a stochastic
  network game: a little empathy goes a long way in averting outbreaks},'' {\em
  Scientific Reports}, vol.~7, no.~July 2016, p.~44122, 2017.

\bibitem{Brauer2006}
F.~Brauer, ``{Some Simple Epidemic Models},'' {\em Mathematical Biosciences and
  Engineering}, vol.~3, no.~1, pp.~1--15, 2006.

\bibitem{Yan2007}
X.~Yan, Y.~Zou, and J.~Li, ``{Optimal quarantine and isolation strategies in
  epidemics control},'' {\em World Journal of Modelling and Simulation},
  vol.~3, no.~3, pp.~202--211, 2007.

\bibitem{Nowzari2016}
C.~Nowzari, V.~M. Preciado, and G.~J. Pappas, ``{Analysis and Control of
  Epidemics},'' {\em IEEE Control Systems Magazine}, vol.~36, no.~1,
  pp.~26--46, 2016.

\bibitem{Simon2017}
P.~Simon and I.~Z. Kiss, ``{On bounding exact models of epidemic spread on
  networks},'' {\em arXiv preprint}, pp.~1--18, 2017.

\bibitem{Watkins2016}
N.~J. Watkins, C.~Nowzari, V.~M. Preciado, and G.~J. Pappas, ``{Optimal
  Resource Allocation for Competitive Spreading Processes on Bilayer
  Networks},'' {\em IEEE Transactions on Control of Network Systems}, vol.~5,
  no.~1, pp.~298--307, 2018.

\bibitem{Dasaklis2012}
T.~K. Dasaklis, C.~P. Pappis, and N.~P. Rachaniotis, ``{Epidemics control and
  logistics operations: A review},'' {\em International Journal of Production
  Economics}, vol.~139, no.~2, pp.~398--410, 2012.

\bibitem{Metcalf2015a}
C.~J.~E. Metcalf, R.~B. Birger, S.~Funk, R.~D. Kouyos, J.~O. Lloyd-Smith, and
  V.~A.~A. Jansen, ``{Five challenges in evolution and infectious diseases},''
  {\em Epidemics}, vol.~10, pp.~40--44, 2015.

\bibitem{Arias2009}
C.~A. Arias and B.~E. Murray, ``{Antibiotic-Resistant Bugs in the 21st Century
  — A Clinical Super-Challenge},'' {\em The New England Journal of Medicine},
  pp.~439--443, 2009.

\bibitem{Allgower2000}
F.~Allg{\"{o}}wer and A.~Zheng, eds., {\em {Nonlinear model predictive
  control}}.
\newblock Boston: Springer Basel AG, 2000.

\bibitem{Ellis2014}
M.~Ellis, H.~Durand, and P.~D. Christofides, ``{A tutorial review of economic
  model predictive control methods},'' {\em Journal of Process Control},
  vol.~24, no.~8, pp.~1156--1178, 2014.

\bibitem{Mesbah2016}
A.~Mesbah, ``{Stochastic Model Predictive Control},'' {\em IEEE Control Systems
  Magazine}, vol.~3, no.~December, pp.~30--44, 2016.

\bibitem{Borelli2017}
F.~Borelli, A.~Bemporad, and M.~Morari, {\em {Predictive Control for linear and
  hybrid systems}}.
\newblock Cambridge University Press, first~ed., 2017.

\bibitem{Hikal2014}
M.~M. Hikal, ``{Dynamic Properties for a General SEIV Epidemic},'' {\em
  Electronic Journal of Mathematical Analysis and Applications}, vol.~2, no.~1,
  pp.~26--36, 2014.

\bibitem{Sahneh2013}
F.~D. Sahneh, C.~Scoglio, and P.~V. Mieghem, ``{Generalized Epidemic Mean-Field
  Model for Spreading Processes Over Multilayer Complex Networks},'' {\em
  IEEE/ACM Transactions on Networking}, vol.~21, no.~5, pp.~1609--1620, 2013.

\bibitem{Geyer2014}
T.~Geyer and D.~E. Quevedo, ``{Multistep finite control set model predictive
  control for power electronics},'' {\em IEEE Transactions on Power
  Electronics}, vol.~29, no.~12, pp.~6836--6846, 2014.

\bibitem{Preindl2013}
M.~Preindl and S.~Bolognani, ``{Model predictive direct torque control with
  finite control set for pmsm drive systems, part 2: Field weakening
  operation},'' {\em IEEE Transactions on Industrial Informatics}, vol.~9,
  no.~2, pp.~648--657, 2013.

\bibitem{Karamanakos2014}
P.~Karamanakos, T.~Geyer, N.~Oikonomou, F.~D. Kieferndorf, and S.~Manias,
  ``{Direct model predictive control: A review of strategies that achieve long
  prediction intervals for power electronics},'' {\em IEEE Industrial
  Electronics Magazine}, vol.~8, no.~1, pp.~32--43, 2014.

\bibitem{Richards2005}
A.~Richards and J.~How, ``{Mixed-integer programming for control},'' {\em
  Proceedings of the 2005, American Control Conference, 2005.}, pp.~2676--2683,
  2005.

\bibitem{Rawlings2017a}
J.~B. Rawlings and M.~J. Risbeck, ``{Model predictive control with discrete
  actuators: Theory and application},'' {\em Automatica}, vol.~78,
  pp.~258--265, 2017.

\bibitem{Rawlings2018}
J.~B. Rawlings, N.~R. Patel, M.~J. Risbeck, C.~T. Maravelias, M.~J. Wenzel, and
  R.~D. Turney, ``{Economic MPC and real-time decision making with application
  to large-scale HVAC energy systems},'' {\em Computers and Chemical
  Engineering}, vol.~114, pp.~89--98, 2018.

\bibitem{Ellis2016}
M.~Ellis, J.~Liu, and P.~Christophides, {\em {Economic Model Predictive
  Control}}.
\newblock Springer, 2016.

\bibitem{Angeli2011}
D.~Angeli, R.~Amrit, and J.~B. Rawlings, ``{Enforcing convergence in nonlinear
  economic MPC},'' {\em Proceedings of the IEEE Conference on Decision and
  Control}, pp.~3387--3391, 2011.

\bibitem{Amrit2011}
R.~Amrit, J.~B. Rawlings, and D.~Angeli, ``{Economic optimization using model
  predictive control with a terminal cost},'' {\em Annual Reviews in Control},
  vol.~35, no.~2, pp.~178--186, 2011.

\bibitem{Subramanian2014}
K.~Subramanian, J.~B. Rawlings, and C.~T. Maravelias, ``{Economic model
  predictive control for inventory management in supply chains},'' {\em
  Computers and Chemical Engineering}, vol.~64, pp.~71--80, 2014.

\bibitem{Amrit2013}
R.~Amrit, J.~B. Rawlings, and L.~T. Biegler, ``{Optimizing process economics
  online using model predictive control},'' {\em Computers and Chemical
  Engineering}, vol.~58, pp.~334--343, 2013.

\bibitem{Angeli2012}
D.~Angeli, R.~Amrit, and J.~B. Rawlings, ``{On average performance and
  stability of economic model predictive control},'' {\em IEEE Transactions on
  Automatic Control}, vol.~57, no.~7, pp.~1615--1626, 2012.

\bibitem{Grune2017}
L.~Gr{\"{u}}ne and J.~Pannek, {\em {Nonlinear Model Predictive Control}}.
\newblock Springer, 2017.

\bibitem{Faulwasser2018}
T.~Faulwasser, L.~Gr{\"{u}}ne, and M.~A. M{\"{u}}ller, ``{Economic Nonlinear
  Model Predictive Control},'' {\em Foundations and Trends in Systems and
  Control}, vol.~5, no.~1, p.~90, 2018.

\bibitem{Rawlings2017b}
J.~B. Rawlings, D.~Q. Mayne, and M.~Diehl, {\em {Model Predictive Control:
  Theory, Computation, and Design}}.
\newblock Nob Hill Publishing, 2017.

\bibitem{Mayne2000}
D.~Q. Mayne, ``{Nonlinear Model Predictive Control: Challenges and
  Opportunites},'' in {\em Nonlinear Model Predictive Control}
  (F.~Allg{\"{o}}wer and A.~Zheng, eds.), ch.~2, pp.~23--44, Boston: Springer
  Basel AG, first~ed., 2000.

\bibitem{Cator2012}
E.~Cator and P.~{Van Mieghem}, ``{Second-order mean-field
  susceptible-infected-susceptible epidemic threshold},'' {\em Physical Review
  E}, vol.~85, no.~5, pp.~1--7, 2012.

\bibitem{Singh2006}
A.~Singh and J.~P. Hespanha, ``{Moment closure techniques for stochastic models
  in population biology},'' {\em American Control Conference}, pp.~6----pp,
  2006.

\bibitem{Singh2011}
A.~Singh and J.~P. Hespanha, ``{Approximate moment dynamics for chemically
  reacting systems},'' {\em IEEE Transactions on Automatic Control}, vol.~56,
  no.~2, pp.~414--418, 2011.

\bibitem{Ghusinga2017}
K.~R. Ghusinga, M.~Soltani, A.~Lamperski, S.~Dhople, and A.~Singh,
  ``{Approximate moment dynamics for polynomial and trigonometric stochastic
  systems},'' {\em arXiv preprint}.

\bibitem{Soltani2014}
M.~Soltani, C.~Vargas, N.~Kumar, R.~Kulkarni, and A.~Singh, ``{Moment Closure
  Approximations in a Genetic Negative Feedback Circuit},'' p.~6, 2014.

\bibitem{Ruschendorf1991}
L.~R{\"{u}}schendorf, ``{Fr{\'{e}}chet-Bounds and Their Applications},'' Tech.
  Rep. January, 1991.

\bibitem{Khalil2002}
H.~K. Khalil, {\em {Nonlinear Systems}}.
\newblock Upper Saddle River, NJ: Prentice Hall, third~ed., 2002.

\bibitem{CDC2017}
CDC/HHS, ``{Final Rule for Control of Communicable Diseases},'' 2017.

\bibitem{Drazen2014}
J.~M. Drazen, R.~Kanapathipillai, E.~W. Campion, E.~J. Rubin, S.~M. Hammer,
  S.~Morrissey, and L.~R. Baden, ``{Ebola and Quarantine},'' {\em New England
  Journal of Medicine}, vol.~371, no.~21, pp.~2029--2030, 2014.

\bibitem{Burer2012}
S.~Burer and A.~N. Letchford, ``{Non-convex mixed-integer nonlinear
  programming: A survey},'' {\em Surveys in Operations Research and Management
  Science}, vol.~17, no.~2, pp.~97--106, 2012.

\bibitem{Lee2012}
J.~Lee, Y.~Sun, and M.~Saunders, ``{Proximal Newton-type methods for convex
  optimization},'' {\em Nips}, pp.~1--9, 2012.

\bibitem{Blum2016}
C.~Blum and G.~R. Raidl, {\em {Hybrid Metaheuristics: Powerful Tools for
  Optimization}}.
\newblock Springer, 2016.

\bibitem{Hansen2001}
P.~Hansen and N.~Mladenovi{\'{c}}, ``{Variable neighborhood search: Principles
  and applications},'' {\em European Journal of Operational Research},
  vol.~130, no.~3, pp.~449--467, 2001.

\bibitem{Hanson2007}
F.~B. Hanson, {\em {Applied Stochastic Process and Control for
  Jump-Diffusions}}.
\newblock Philadelphia: Society for Industrial and Applied Mathematics, 2007.

\bibitem{Rudin1976}
W.~Rudin, {\em {Principles of Mathematical Analysis}}.
\newblock McGraw-Hill Education, 3rd~ed., 1976.

\bibitem{Smith2004}
M.~W. Hirsch and H.~Smith, {\em {Monotone Dynamical Systems}}.
\newblock 2006.

\bibitem{Rosenthal2013}
J.~S. Rosenthal, {\em {A First Look at Rigorous Probability Theory}}.
\newblock Hackensack, NJ: World Scientific Publishing Co., 2nd~ed., 2013.

\end{thebibliography}
\bibliographystyle{ieeetr}
\end{document}